\documentclass{amsart}

\usepackage{graphicx}
\usepackage{bbm}
\usepackage{amsmath}
\usepackage{amssymb}
\usepackage{amsthm}
\usepackage{enumerate} 
\usepackage{enumitem}
\usepackage{color}
\usepackage[titletoc]{appendix}
\usepackage{multirow}
\usepackage{soul}

\newcommand{\Z}{\Bbb Z}

\newcommand{\C}{\Bbb C}

\newcommand{\g}{  \mathfrak g}

\newcommand{\bea}{\begin{eqnarray}}
\newcommand{\eea}{\end{eqnarray}}

\newtheorem{theorem}{Theorem}[section]
\newtheorem{lemma}[theorem]{Lemma}
\newtheorem{proposition}[theorem]{Proposition}
\newtheorem{definition}[theorem]{Definition}

\newtheorem{remark}[theorem]{Remark}
\newtheorem{corollary}[theorem]{Corollary}
\newtheorem{conj}[theorem]{Conjecture}

\makeatletter
\newtheorem*{rep@theorem}{\rep@title}
\newcommand{\newreptheorem}[2]{%
	\newenvironment{rep#1}[1]{%
		\def\rep@title{#2 \ref{##1}}%
		\begin{rep@theorem}}%
		{\end{rep@theorem}}}
\makeatother
\newreptheorem{lemma}{Lemma}
\newreptheorem{definition}{Definition}
\newreptheorem{proposition}{Proposition}
\newreptheorem{theorem}{Theorem}
\numberwithin{equation}{section}

\newcommand{\Addresses}{{% additional braces for segregating \footnotesize
  \bigskip
  \footnotesize

  D. Adamovi\' c, \textsc{ Department of Mathematics, University of Zagreb,  Croatia}\par\nopagebreak
  \textit{E-mail address:} \texttt{adamovic@math.hr}

  \medskip

  A. Kontrec, \textsc{Department of Mathematics, University of Zagreb,  Croatia}\par\nopagebreak
  \textit{E-mail address:} \texttt{ana.kontrec@math.hr}

}}

\begin{document}

\title[Bershadsky-Polyakov algebra at positive integer levels]{  Bershadsky-Polyakov  vertex algebras    at positive integer levels and duality}
%\title[Realization]{On duality of   $L_{k} (osp(1 \vert 2))$ and $\mathcal{W}_k(sl_3,f_{\theta})$ }

\author{Dra\v zen Adamovi\' c}

%\address{Department of Mathematics}

\author{Ana Kontrec}

\begin{abstract}

We study the simple   Bershadsky-Polyakov algebra $\mathcal W_k = \mathcal{W}_k(sl_3,f_{\theta})$  at positive integer levels and classify their irreducible modules. In this way we confirm the conjecture from \cite{AK-2019}.
Next, we study the case  $k=1$. We discover that this vertex algebra has a Kazama-Suzuki-type dual isomorphic to the simple affine vertex superalgebra $L_{k'} (osp(1 \vert 2))$ for $k'=-5/4$.
%
  % in the sense that  $\mathcal W_k \subset L_{k'} (osp(1 \vert 2)) \otimes F$ and $L_{k'} (osp(1 \vert 2)) \subset \mathcal W_k \otimes F_{-1}$, where $F$ is a $bc$ system, and $F_{-1}$ is a lattice type vertex algebra. 
   Using  the free-field realization of  $L_{k'} (osp(1 \vert 2))$  from \cite{A-2019}, we get a free-field realization of $\mathcal W_k$ and their highest weight modules.  In a sequel, we plan to study fusion rules for $\mathcal W_k$.
\end{abstract}
\date{\today}
\maketitle
% \section{Irreducible $\mathcal{W}_k$-modules for integer levels $k$
  
  \section{Introduction}

  In the recent years,  minimal affine $\mathcal W$ algebras have attracted a lot of interest.   They  are obtained using quantum hamiltonian reduction from affine vertex algebras, and they can be described using generators and relations (cf. \cite{KW, KRW}). 

The Bershadsky-Polyakov algebra $\mathcal W_k = \mathcal{W}_k(sl_3,f_{\theta})$ (\cite{Ber}, \cite{Pol}) is the simplest minimal affine $\mathcal W$--algebra.  T. Arakawa proved in \cite{A1}  that $\mathcal W_k$ is rational for  $k + 3/2 \in {\Z}_{\ge 0}$, while in other cases it is a non-rational vertex algebra. More recently, for $k$ admissible and non-integral, irreducible $\mathcal W_k$--modules were classified in \cite{AK-2019} in some special cases, and in \cite{KRZ} in full generality. A realization of $\mathcal W_k$, when $2k +3 \notin {\Z_{\ge 0}}$, and its relaxed modules is presented in \cite{AKR-2020}, which gives a natural  generalization of the realization of the affine vertex algebra $V_k(sl(2))$ from \cite{A-2019}. Let us now mention certain problems for Bershadsky-Polyakov vertex algebras, which remain unsolved in papers listed above.

\subsection*{ A. Classification of irreducible $\mathcal W_k$--modules for integer levels $k$, $k+ 2 \in {\Z}_{\ge 0}$. }
   In \cite{KRZ}, authors classified irreducible  $\mathcal W_k$--modules for $k$ admissible, non-integral. They showed that every irreducible highest weight module for $\mathcal W_k$ is obtained as an image of the admissible modules for $L_k(sl(3))$ (which are classifed by T. Arakawa in \cite{Araduke}). 
However, when we pass to integral $k$, the methods of \cite{KRZ} are no longer applicable, since in this case quantum hamiltonian reduction sends $L_k(sl(3))$--modules to zero. 
 
 In \cite{AK-2019},  we began the study of the representation theory of $\mathcal W_k$ for integer levels $k$, $k+ 2 \in {\Z}_{\ge 0}$. The starting point was explicit formulas for singular vectors in $\mathcal W^k$ (which generalized those of Arakawa in \cite{A1}).
 We presented a conjecture on the classification of irreducible $\mathcal W_k$--modules for $k+ 2 \in {\Z}_{\ge 0}$, which we proved in cases $k=-1,0$ (and classified all modules in the category $\mathcal{O}$) using explicit realizations of $\mathcal W_k$. One of the main results in this paper is the proof of this conjecture.
 
\subsection*{B. Free-field realization of $\mathcal W_k$ and their modules for $k+ 2 \in {\Z}_{\ge 0}$.}

    In \cite{AKR-2020}, the Bershadsky-Polyakov algebra $\mathcal W_k$ is realized as a vertex subalgebra of $Z_k \otimes \Pi(0)$  (where $Z_k =W_k(sl(3), f_{princ})$ is the Zamolodchikov $W$-algebra \cite{Zam}), for $2 k +3 \notin {\Z}_{\ge 0}$. A realization of $W_k$, when $2k +3 \in {\Z}_{\ge 0}$, requires a different approach. In \cite{AK-2019} we constructed a free-field realization of $\mathcal W_0$, but the cases when $k >0$ were not solved either in \cite{AK-2019} or \cite{AKR-2020}.

In the current  paper, we continue with our study of the Bershadsky-Polyakov algebra $\mathcal W_k = \mathcal{W}_k(sl_3,f_{\theta})$  at positive integer levels $k$   and completely solve problem (A) for $k \ge 0$. We also  partially solve the problem (B) for $k=1$ and find duality relation of $\mathcal W_1$ with the affine vertex superalgebra associated to $osp(1\vert 2)$.  

\subsection*{Classification of irreducible representations}

In \cite{AK-2019}, we found a necessary condition for $\mathcal{W}_k$--modules, parametrizing the highest weights as zeroes of certain polynomial functions (cf. Proposition \ref{klas-1})
\begin{align*} 
h_i(x,y) &= \frac{1}{i}(g(x,y) + g(x+1,y) + ... + g(x+i-1,y)) &\\
&= -i^2+ki-3xi+3i-3x^2-k+2kx+6x+ky+3y-2,&
\end{align*} 
and conjectured that this provides the complete list of irreducible modules for the Bershadsky-Polyakov algebra $\mathcal W_k$ when $k \in \mathbb{Z}$, $k \geq -1$. This conjecture was proved in cases $k=-1$ and $k=0$ in \cite{AK-2019}, using explicit realizations of $\mathcal W_{-1}$ and $\mathcal W_0$ as the Heisenberg vertex algebra and a subalgebra of lattice vertex algebra, respectively. In this paper, we prove this conjecture for $k \in\mathbb{Z}$, $k \geq 1$, thus obtaining a classification of irreducible modules for $\mathcal W_k$ at positive integer levels.
 
 \begin{theorem} \label{klas}
The set   $ \{ L(x,y) \ \vert  \ (x,y) \in  \mathcal S_k \} $	 is the set of all irreducible ordinary $\mathcal W_k$--modules, where $$  \mathcal S_k= \left\{  (x,y)  \in {\C} ^ 2 \ \vert  \exists i , \ 1 \leq i \leq k+2, \  h_i(x,y) = 0  \right\}. $$
\end{theorem}

In order to prove Theorem \ref{klas}, we need to show that   $L(x,y)$, $(x,y) \in  \mathcal S_k$,  are indeed $\mathcal W_k$--modules. Idea of the proof is to  construct an infinite family of irreducible $\mathcal W_k$--modules $L(x,y)$ such that  $h_i(x,y) = 0$ for arbitrary $1 \le i \le k+2$, using  spectral flow construction of  $\mathcal W_k$--modules (cf. Section \ref{integer}). 

\begin{itemize}
\item We first consider a family of simple-current $\mathcal W_k$--modules
$\Psi^n (\mathcal W_k)$, $n \in {\Z}_{\ge 0}$. 
\item  We show that they are highest  weight $\mathcal W_k$--modules satisfying
$$ \Psi^n (\mathcal W_k) = L(x_{n}, y_{n}) \quad \text{and} \quad h_1 (x_{2n}, y_{2n}) =h_{k+2} (x_{2n+1}, y_{2n+1})=   0. $$
\item Using a certain version of algebraic continuation (based on the fact that highest weights of  modules for Zhu's algebra must be  zeros of finitely  many curves in ${\C}^2$), we conclude that
$L(x,y)$ are $\mathcal W_k$--modules whenever $h_1(x,y) =0$ or $h_{k+2}(x,y) = 0$.

\item Next, for every $2 \le i \le k+1$, we find special points $(x^i , y^i)$ such that 
$h_{i}(x^i,y^i) = h_{k+2}(x^i,y^i) = 0$, and again apply the spectral-flow automorphism $\Psi^n$. In this way  we are able to construct infinitely many  highest weight $\mathcal W_k$--modules  $L(x^i_{2n}, y^i_{2n})$ such that $h_i(x^i_{2n}, y^i_{2n}) =0$.

\item Again using the algebraic continuation, we conclude that $L(x,y)$ are $\mathcal W_k$--modules for each point of the curve $h_i(x,y) =0$.

\end{itemize}

% Then, from the properties of the Zhu algebra associated to $\mathcal W_k$, it follows that all $L(x,y)$ satisfying this condition are $\mathcal W_k$--modules.

\subsection*{ Realization of $\mathcal W_1$ and duality with $L_{-5/4}(osp(1 \vert 2))$}

Next, we give an in-depth study of the case $k=1$. First we show that the Bershadsky-Polyakov vertex algebra $\mathcal W_1$  can be embedded into the tensor product of the affine vertex {super}algebra $L_{k'} (osp(1 \vert 2))$ at level $k'=-5/4$ and the Clifford vertex {super}algebra $F$.
The affine vertex algebra  $V^{k}(osp(1 \vert 2)) $ associated to the Lie superalgebra $osp(1,2)$ was realized by the first named author in \cite{A-2019}. Using this result, and the fact that at level $k'=-5/4$ there is a conformal embedding of $L_{k'}(sl(2))$ into $L_{k'}(osp(1 \vert 2))$ (cf.  \cite{AKMPP-2020}, \cite[Section 10]{CR}), we obtain a realization of the Bershadsky-Polyakov algebra $\mathcal W_1$ (cf. Theorem \ref{real-1}).

Let $F_{-1}$ be the  lattice   vertex {super}algebra associated to the negative definite lattice ${\Bbb Z}\sqrt{-1}$. We show that the simple affine vertex {super}algebra $L_{-5/4} (osp(1 \vert 2))$ can be realized as a subalgebra of $\mathcal W_1 \otimes F_{-1}$ (cf. Theorem \ref{duality}). Moreover, there is a duality   between $\mathcal W_1$ and the affine vertex {super}algebra $L_{k'} (osp(1 \vert 2))$ for $k'=-5/4$, {  in the sense that:
 \bea \mathcal W_1 &=& \mbox{Com} \left(M_{h^{\perp}}(1),  L_{k'} (osp(1 \vert 2)) \otimes F\right), \nonumber \\
  L_{k'} (osp(1 \vert 2)) &=& \mbox{Com} \left( M_{\overline h} (1),  \mathcal W_1 \otimes F_{-1}\right),
  \nonumber \eea
  where $M_{h^{\perp}}(1)$ and $M_{\overline h} (1)$ are Heisenberg vertex algebras defined in Section \ref{duality-section}. }
 
 In \cite{A-2019} it was proved that $L_{-5/4}(osp(1 , 2))$ can be realized on the vertex {super}algebra $F^{1/2} \otimes \Pi^{1/2}(0)$, where $\Pi^{1/2}(0)$ is a lattice type vertex algebra, and $F^{1/2}$ is a Clifford vertex {super}algebra (cf. Subsection \ref{clifford-pola}). Using the fact that all irreducible $L_{-5/4}(osp(1 \vert 2))$-modules can be constructed in this way,  we can construct an explicit realization of irreducible $\mathcal W_1$-modules.

\subsection*{Consequences of duality and future work}

 The notion of Kazama-Suzuki dual was first introduced in the context of the duality of the $N=2$ superconformal algebra and affine Lie algebra $\widehat{\frak sl}(2)$ (cf.  \cite{FST, A-IMRN, A-2001}). Later it was shown that analogous duality relations hold for some other affine vertex algebras and $\mathcal W$--algebras (cf. \cite{A-2007, A-CEJM, CGN, CMY}). Our result shows that  $L_{k'} (osp(1 \vert 2))$ is the Kazama-Suzuki dual of $\mathcal W_1$. Relaxed modules for $L_{k'}(osp(1 \vert 2))$ are mapped  to the ordinary $\mathcal W_1$--modules, for which  one expects it is easier to obtain the tensor category structure and calculate the fusion rules.  Recent results \cite{AP-2019, CMY} show compelling evidence that fusion rules and (vertex) tensor category structure can be transferred onto duals. We expect that the duality between $L_{k'} (osp(1 \vert 2))$ and $\mathcal W_1$ could be used to study fusion rules in the category of relaxed modules for $L_{k'}(osp(1 \vert 2))$ (conjectured in \cite{SRW}).
 %and to establish tensor category structure for certain module categories. 

%One can expect that these fusion rules can be proved much easier in the context of the representation theory of $\mathcal W_1$.

%In our future work, we plan to study this and related problems.

{
\subsection*{ Setup} 
\begin{itemize}
\item  The universal Bershadsky-Polyakov algebra of level $k$  will be denoted with $\mathcal{W}^k(sl_3,f_{\theta})$ or $\mathcal W^k$, and its unique simple quotient with $\mathcal{W}_k(sl_3,f_{\theta})$ or  $\mathcal W_k$.  The spectral flow automorphism of $\mathcal W_k$ is denoted by $\Psi$.
%\item Modes of fields with respect to Virasoro vector $\omega$: $J_n$, $L_n$, $G^{\pm}_n$.
%\item $L_{x,y}$ denotes the irreducible highest weight representation with highest weight $(x,y)$  with respect to $(J_0, L_0)$ and highest weight vector $v_{x,y}$.
%\item $L_{x,y} ^*$ denotes the module contragredient to $L_{x,y}$ with respect to the Virasoro algebra generated by $L_n, n \in \mathbb{Z} $
%\item Modes of fields with respect to Virasoro vector $\overline \omega=\omega + \frac{1}{2} DJ$: $$J(n)= J_n, \ L(n) = L_n - \frac{n+1}{2} J_n,  \  G^+ (n) = G^+_n,  \ G^-  (n) = G^- _ {n+1}.  $$
%\item $L(x,y) $ denotes the irreducible highest weight representation with highest weight $(x,y)$  with respect to $(J(0), L(0) )$ and highest weight vector $v(x,y)$.
%\item We have $L(x,y) =L_{x, y+x/2}$.
\item The Zhu algebra associated to the vertex operator algebra $V$ with the Virasoro vector $\omega$  will be denoted with $A_{\omega}(V)$. 
\item The Smith algebra corresponding to the polynomial $g(x,y) \in \mathbb{C}[x,y]$ is denoted with $R(g)$.
\item $F_{-1}$ is the lattice vertex superalgebra  associated to the lattice ${\Z} \sqrt{-1}$ defined in Section \ref{F_{-1}}.
\item $F^{1/2}$ is the Clifford vertex superalgebra, also called the   free fermion algebra (cf. Section \ref{clifford-pola}). It has an automorphism $\sigma_{F^{1/2}}$ of order two which is lifted from the automorphism $\Phi(r) \mapsto - \Phi(r)$ of the Clifford algebra $Cl^{1/2}$. The $\sigma_{F^{1/2}}$--twiseted $F^{1/2}$--modules are denoted by $M_{F} ^{\pm}$.
\item $F$ is the Clifford vertex superalgebra, also called the charged fermion algebra or the  $bc$ system (cf. Section \ref{clifford-cijeli}). It has an automorphism $\sigma_{F}$ of order two which is lifted from the automorphism $\Psi^+(r) \mapsto -  \Psi^+(r)$, $\Psi^-(r) \mapsto -  \Psi^-(r)$ of  the Clifford algebra $Cl$.  The $\sigma_F$--twisted
$F$ module is denoted by $M_{F} ^{tw}$.
\item $L_k(osp(1|2))$ is the simple affine vertex superalgebra associated to the Lie superalgebra $osp(1|2)$ at level $k$. The spectral flow automorphism of $L_k(osp(1|2))$ is denoted by $\rho$.

\end{itemize}}

\section{ Preliminaries} \label{clif} 
In this section we review certain properties of Clifford vertex superalgebras (cf. \cite{FRW}, \cite{FFR}) and a construction of twisted modules for vertex superalgebras by H. Li (cf. \cite{Li3}). Twisted modules for Clifford vertex superalgebras (cf. \cite{FRW}) will play a key role in the realization of the Bershadsky-Polyakov algebra $\mathcal{W}_1$.
%Clifford vertex algebra $F$ is the universal vertex algebra  generated by the odd fields  $\Psi^{+}$ and $\Psi^{-}$ , and the following $\lambda$--brackets: $$ [\  \Psi^+  _{\lambda} \Psi^- \  ] =  1, \quad [\  \Psi^{\pm}  _{\lambda} \Psi^{\pm} \  ] =  0. $$ 

\subsection{Twisted modules for vertex superalgebras}
 %A \textit{vertex operator algebra} (VOA)  (cf. \cite{FLM,LL})  
 
Let $V = V_{\overline 0} \oplus V_{\overline 1}$ be a vertex  superalgebra (cf. \cite{FFR}, \cite{X}), {with the vertex operator structure given by $$ Y: V \longrightarrow (\text{End} \, V)[[z,z^{-1}]], \quad Y(v,z)= \sum _{n \in \mathbb{Z}} v_{n}z^{-n-1},$$ for $v \in V$, $v_n \in \text{End} \, V$.}  Then any element in $V_{\overline 0}$ (resp. $V_{\overline 1}$) is said to be even (resp. odd). For any homogenuous element $u$, we define $|u| = 0$ if $u \in V_{\overline 0}$ and $|u| = 1$ if $u \in V_{\overline 1}$.
 %For $\alpha \in \left \{ \overline{0}, \overline{1} \right \} $, let $|v| = \alpha$ if  $v \in V_{\alpha}$. 
 
We say that a linear automorphism $\sigma:V \rightarrow V$ is a \textit{vertex superalgebra automorphism} if it holds that
	$$ \sigma Y(v,z) \sigma^{-1} = Y(\sigma v, z)$$ for $v \in V$. Then $\sigma  V_{\alpha} \subset  V_{\alpha}$ for $\alpha \in \left \{ \overline{0}, \overline{1} \right \}$.

\medskip

Let $V$ be a vertex  superalgebra and $\sigma$ an automorphism of $V$ with period $k \in \mathbb{Z}_{\geq 0}$ (that is, $\sigma^k = 1$). Let us now recall a construction of $\sigma$-twisted $V$-modules (cf. \cite{Li3}). 

%\subsubsection*{A construction of twisted modules} 

Let $h \in V$ be an even element such that
\begin{equation} \label{delta_1}
 L(n)h = \delta_{n,0}h, \: h_n h= \delta_{n,1} \gamma \mathbbm{1} \text{ for } n \in \mathbb{Z}_{>0},
\end{equation} for fixed $\gamma \in \mathbb{Q}$. Assume that $h_0$ acts semisimply on  $V$ with rational eigenvalues. It follows that $h_n$ satisfy:
$$ [L(m),h_n] = -n h_{m+n}, \quad [h_m,h_n] = m\gamma \delta_{m+n,0},$$ for $m, n \in \mathbb{Z}. $

Set $$ \Delta (h,z) = z^{h_0}exp \left ( \sum _{k=1} ^{\infty} \frac{h_k}{-k} (-z)^{-k}\right ). $$ Note that $e^{2 \pi ih_0}$ is an automorphism of $V$. Set $\sigma_h = e^{2 \pi ih_0}$ and assume that $\sigma_h$ is of finite order. The following was proved in  \cite{Li3}:

\begin{proposition}[\cite{Li3}] \label{li-twisted}
Let  $V$ be a vertex superalgebra and let $h \in V$ be an even element such that (\ref{delta_1})
holds and $h_0$ acts on $V$ with rational eigenvalues. Let $(M, Y_M(\cdot, z))$ be a $V$-module. Then $(M, Y_M(\Delta(h,z)\cdot, z))$ carries the structure of a $\sigma_h$-twisted $V$-module.
\end{proposition}

\subsection{Clifford vertex superalgebra $F^{1/2}$ and its twisted modules}
\label{clifford-pola}
 
Let  $\textit{Cl}^{1/2}$ be the Clifford algebra   with generators 
 $\Phi  (r),  \ \ r \in \tfrac{1}{2} + {\Z}$  and commutation relations
$$\{\Phi (r), \Phi (s)  \} = \delta_{r+s,0}, \quad r,s \in \tfrac{1}{2} + {\Z}. $$ 

The fields $$\Phi(z)= \sum _{n \in  \mathbb{Z}} \Phi(n+ \tfrac{1}{2}) z^{-n-1}$$ generate on $$F^{1/2}   = \bigwedge \left( \Phi (-n-1/2)\vert  \ n  \in {\Z}_{\geq 0}  \right)$$ a unique structure of a vertex  superalgebra with conformal vector $$ \omega_{F^{1/2}} = \tfrac{1}{2} \Phi(-\tfrac{3}{2}) \Phi(-\tfrac{1}{2}) \mathbbm{1}, $$  of central charge $c_{F^{1/2}}=1/2$ (cf. \cite{FB}, \cite{K1}). { Note also that the field $\Phi(z)$ is usually called neutral fermion field, and $F^{1/2}$ is called free-fermion theory in physics literature.}

A basis of $F^{1/2}$ is given by $$ \Phi(-n_1 - \tfrac{1}{2})\dots \Phi(-n_r - \tfrac{1}{2}),$$ where  $n_1>\dots >n_r \geq 0$.

The vertex superalgebra $F^{1/2}$ has the automorphism $\sigma_{F^{1/2}}$ of
 order two which is lifted from the automorphism $\Phi(r) \mapsto - \Phi(r)$ of  the Clifford algebra. The fixed points of this automorphism is the Virasoro vertex algebra $L^{Vir}(\tfrac{1}{2}, 0)$. Moreover,
 $F^{1/2} = L^{Vir}(\tfrac{1}{2}, 0) \oplus L^{Vir}(\tfrac{1}{2}, \frac{1}{2})$.

%}

We briefly recall the properties of twisted modules for Clifford vertex superalgebras, while details can be found in \cite{FRW}.

Define the twisted Clifford algebra $\textit{Cl}^{1/2}_{tw}$ generated by
 $\Phi  (m),  \ \ m \in  {\Z}$  and relations
$$\{\Phi (m), \Phi (n)  \} = \delta_{m+n,0}, \quad m,n \in  {\Z}. $$  
Let  $$M^{\pm}_{F^{1/2}} = \bigoplus_{n=0}^{\infty}M^{\pm}_{F^{1/2}}(n)  $$   be the two irreducible modules for the Clifford  algebra $\textit{Cl}^{1/2}_{tw}$, such that  $\Phi(0)$ acts on the one-dimensional top component $M^{\pm}_{F^{1/2}}(0) $ as $\pm \frac{1}{\sqrt{2}}\text{Id}$. 
 
Let $$ \Phi^{tw}(z)=  \sum _{m \in  \mathbb{Z}} \Phi(m) z^{-m-1/2}, $$ and $$ Y(\Phi(-n_1 - \tfrac{1}{2})\dots \Phi(-n_r - \tfrac{1}{2}),z) = : \partial_{n_1}\Phi^{tw}(z_1)\dots\partial_{n_r}\Phi^{tw}(z_r):,  $$ and extend by linearity to all of $F^{1/2}$.

Define the twisted operator $$Y_{F^{1/2}}^{tw}(v,z) := Y(e^{\Delta_z} v,z),$$ where $$\Delta_z = \frac{1}{2} \sum_{m,n \in \mathbb{Z}_{\geq 0}}C_{m,n}\Phi(m+\tfrac{1}{2})\Phi(n+\tfrac{1}{2})z^{-m-n-1},$$ and $$C_{m,n} = \frac{1}{2} \frac{m-n}{m+n+1} \binom{-1/2}{m} \binom{-1/2}{n}.$$ It holds that (cf. \cite{FFR}, \cite{FRW}) $$e^{\Delta_z}\omega_{F^{1/2}} = \omega_{F^{1/2}} + \tfrac{1}{16}z^{-2}\mathbbm{1}.$$

Then $(M^{\pm}_{F^{1/2}}, Y_{F^{1/2}}^{tw})$ has the structure of a $\sigma_{F^{1/2}}$-twisted module for the vertex superalgebra $F^{1/2}$.

Recall  also that  as a  $L^{Vir}(\tfrac{1}{2}, 0)$--module, we have (cf. \cite{FRW})  
$$M^{\pm}_{F^{1/2}}  \cong L^{Vir}(\tfrac{1}{2}, \tfrac{1}{16}). $$

% with the module structure given by $$Y_{tw, \pm}(\Phi, z)=  \sum _{m \in  \mathbb{Z}} \Phi(m) z^{-m-1/2}, $$

%Let $\sigma_{F^{1/2}}$ be the canonical automorphism of $F^{1/2}$ of order two. The vertex algebra $F^{1/2}$ has two nonisomorphic irredicible $\sigma_{F^{1/2}}$-twisted modules $M^{\pm}_{F^{1/2}}$, such that  $\Phi(0)$ acts on the top component of $M^{\pm}_{F^{1/2}}$ as $\pm \frac{1}{\sqrt{2}}\text{Id}$.

\subsection{Clifford vertex superalgebra $F$ and its twisted modules}
\label{clifford-cijeli}
 
 Consider the Clifford algebra $\textit{Cl}$  with generators 
 $\Psi^{\pm}  (r),  \ \ r \in \tfrac{1}{2} + {\Z}$  and relations
$$\{\Psi^{+} (r), \Psi^{-} (s)  \} = \delta_{r+s,0}, \quad \{\Psi^{\pm} (r), \Psi^{\pm} (s)  \} = 0, \quad r,s \in \tfrac{1}{2} + {\Z}.$$ 

The fields $$\Psi^{\pm}(z)= \sum _{n \in  \mathbb{Z}} \Psi^{\pm}(n + \tfrac{1}{2}) z^{-n-1}$$ generate on $$F   = \bigwedge \left( \Psi^{\pm} (-n -1/2) \vert   \ n  \in {\Z}_{>0}  \right)$$ a unique structure of a simple vertex superalgebra.  {This vertex algebra is sometimes called $b c$--system.}

Let  $\alpha = :\Psi^+ \Psi^- :$. Then $$ \omega_F = \frac{1}{2} :\alpha \alpha:  $$ is a conformal vector for $F$ of central charge $c_F=1$.

A basis of $F$ is given by $$ \Psi^+(-n_1 - \tfrac{1}{2})\dots \Psi^+(-n_r - \tfrac{1}{2}) \Psi^-(-k_1 - \tfrac{1}{2})\dots \Psi^-(-k_s - \tfrac{1}{2}),$$ where $n_i, k_i \in \mathbb{Z}_{\geq 0}$, $n_1>\dots >n_r $, $k_1>\dots >n_s $.

\medskip

The vertex superalgebra $F$ has an automorphism $\sigma_{F}$ of
 order two which is lifted from the automorphism $\Psi^+(r) \mapsto -  \Psi^+(r)$, $\Psi^-(r) \mapsto -  \Psi^-(r)$ of  the Clifford algebra.
 
 %{\color{red} It holds that $F = F^{1/2} \otimes F^{1/2}$. Define the Clifford vertex algebra $\textit{Cl}_{tw}$ generated by $\Phi  (r),  \ \ r \in  1/2 + {\Z}$  and relations $$\{\Phi (r), \Phi (s)  \} = -(-1)^{2r}\delta_{r+s,0}, \quad r,s \in 1/2 +{\Z}. $$

%Let $\Phi^{(1)}(z)$, $\Phi^{(2)}(z)$ be the fields satisfying non-trivial $\lambda$-brackets  $\Phi^{(1)}_{\lambda}\Phi^{(1)} = \mathbbm{1}$ and $\Phi^{(2)}_{\lambda}\Phi^{(2)} = \mathbbm{1}$, so that $$ \Psi^+(z)= \frac{\Phi^{(1)}(z)  + i\Phi^{(2)}(z)}{\sqrt{2}}, \quad \Psi^-(z) = \frac{\Phi^{(1)}(z)  - i\Phi^{(2)}(z)}{\sqrt{2}}. $$

% Then the automorphism sending $\Phi^{(1)} \mapsto -\Phi^{(1)}$, $\Phi^{(2)} \mapsto -\Phi^{(2)}$ induces an automorphism $\sigma_{F}$ of $F$ given by $\sigma_{F}(\Psi^+) =  -\Psi^+$, $\sigma_{F}( \Psi^-) =  -\Psi^-$.

%Let $M^{\pm}_{F} = M^{\pm}_{F^{1/2}} \otimes M^{\pm}_{F^{1/2}}$ be the two irreducible modules for the Clifford vertex algebra $\textit{Cl}_{tw}$, and}.

 Let   $(M_F ^{tw}, Y^{tw} (\cdot, z))$ so that $M_F ^{tw} = F$ as a vector space, and the vertex operator is
defined by 
 $$Y^{tw}_F(v,z)= Y(\Delta(\alpha/2 ,z) v,z).$$
 
By Proposition \ref{li-twisted} we have that $(M^{tw}_{F}, Y^{tw}_F)$ has the structure of a $\sigma_{F}$-twisted module for the vertex superalgebra $F$.

%We define $$\Delta_z = \frac{1}{2} \sum_{m,n \in \mathbb{Z}_{\geq 0}}C_{m,n}\Phi(m+\tfrac{1}{2})\Phi(n+\tfrac{1}{2})z^{-m-n-1},$$ where $$C_{m,n} = \frac{1}{2} \frac{m-n}{m+n+1} \binom{-1/2}{m} \binom{-1/2}{n}.  $$

%Let $\Theta_F$ be the involution of $F$ induced from the automorphism of the Clifford algebra $Cl$, given by $ \Theta_F(\Psi^\pm) = \Psi^\mp $. Then the vertex algebra $F$ has two nonisomorphic irredicible $\Theta_F$-twisted modules $M^{\pm}_F$, such that  $\Psi^\pm(0)$ acts on the top component of $M^{\pm}_F$ as $\pm \frac{1}{\sqrt{2}i}\text{Id}$.

\subsection{ Lattice vertex superalgebras $F_{-1}$} \label{F_{-1}}

Consider rank one lattice $L= {\Z}\varphi$, $\langle \varphi, \varphi \rangle = -1$. Let $F_{-1}$be the associated vertex algebra. This vertex superalgebra is used for a construction of the inverse of Kazama-Suzuki functor in the context of duality between affine $\widehat{sl}(2)$ and $N=2$ superconformal algebra (cf.  \cite{FST}, \cite{A-IMRN}, \cite{A-2001}). 

As a vector space $F_{-1} = {\C}[L] \otimes M_{\varphi} (1)$, where  ${\C}[L]$ is a group algebra of $L$, and $M_{\varphi} (1)$ the Heisenberg vertex algebra generated by the Heisenberg field $\varphi(z) = \sum_{n \in {\Z} } \varphi(n) z^{-n-1}$ such that
$$ [\varphi(n), \varphi(m)] =   -n \delta_{n+m, 0}. $$

$F_{-1}$ is generated by $e^{\pm \varphi}$. We shall need the relations
\bea e^{\pm \varphi}_n  e^{\pm \varphi} &=&  0 \quad \mbox{for} \ n \ge 1, \nonumber \\
e^{\pm \varphi}_{-m}   e^{\pm \varphi} &=&  S_m(\pm \varphi) e^{2 \varphi}  \quad \mbox{for} \ m \ge 0, \nonumber \\
e^{  \varphi}_n  e^{- \varphi} &=&  0 \quad  \mbox{for} \ n \ge -1, \nonumber \\
e^{ \varphi}_{-m-2}   e^{ \varphi} & =& S_m(\varphi)  \quad   \mbox{for} \ m \ge 0, \nonumber 
\eea
where $ S_m(\varphi):=S_m(\varphi(-1), \varphi(-2), \cdots )$ is the $m$-th Schur polynomial in variables $\varphi(-1), \varphi(-2), \cdots$.

\subsection{Kazama-Suzuki duality}
In this subsection, we will define a duality of vertex algebras which is motivated by the duality between $N=2$ superconformal vertex algebra and affine vertex algebra $L_k(sl(2))$.

{
Recall first  that if $S$ is a vertex subalgebra of $V$, we have the commutant subalgebra of $V$ (cf. \cite{LL}):  $$\mbox{Com} (S, V):= \{ v \in V \ \vert \ a_n v = 0,  \ \forall a \in S, \ \forall n \in {\Z}_{\ge 0}\}. $$

Assume that $U, V$ are vertex superalgebras. We say that $V$ is the Kazama-Suzuki dual of $U$ if there exist injective homomorphisms of vertex superalgebras
$$ \varphi_1 : V \rightarrow U \otimes F, \ \varphi_2 : U \rightarrow V \otimes F_{-1}, $$
so that
$$ V \cong \mbox{Com} \left( \mathcal H^1, U \otimes F\right),  \quad  U \cong \mbox{Com} \left( \mathcal H^2, V \otimes F_{-1}\right),$$
where  $\mathcal H^1$ (resp. $\mathcal H^2$) is a   rank one  Heisenberg vertex subalgebra of
$U\otimes F$ (resp. $V\otimes F_{-1}$).

}

\section{Bershadsky-Polyakov algebra $\mathcal{W}_k(sl_3,f_{\theta})$}

Bershadsky-Polyakov vertex algebra $\mathcal{W}^k (=\mathcal{W}^k(sl_3,f_{\theta}))$ is the minimal affine $\mathcal{W}$-algebra associated to the minimal nilpotent element $f_{\theta}$ (cf. \cite{A1}, \cite{GK}, \cite{KRW}, \cite{KW}, \cite{K-Phd}). The algebra $\mathcal{W}^k$ is generated by four fields $T, J,G^+,G^-$, of conformal weights $2, 1, \frac{3}{2}, \frac{3}{2}$ and is a $\frac{1}{2}\mathbb{Z}$-graded VOA.
\begin{definition}
	Universal Bershadsky-Polyakov vertex algebra $\mathcal{W}^k $  is the vertex algebra generated by fields $T ,J,G^+,G^-$, which satisfy the following relations: 
	\begin{align*}
    J(x)J(y) &\sim \frac{2k+3}{3}(z-w)^{-2}, \enskip  G^{\pm}(z)G^{\pm}(w) \sim0, &\\
    J(z)G^{\pm}(w) &\sim \pm G^{\pm}(w)(z-w)^{-1},&\\
	T(z)T(w) &\sim - \frac{c_k}{2}(z-w)^{-4}+2T(w)(z-w)^{-2}+DT(w)(z-w)^{-1},   &\\
	T(z)G^{\pm}(w) &\sim \frac{3}{2} G^{\pm}(w)(z-w)^{-2}+ DG^{\pm}(w)(z-w)^{-1}, &\\
	T(z)J(w) &\sim  J(w)(z-w)^{-2}+DJ(w)(z-w)^{-1}, &\\
	G^+(z)G^-(w) &\sim (k+1)(2k+3)(z-w)^{-3}+3(k+1)J(w)(z-w)^{-2}+ &\\
	    &+ (3:J(w)J(w): + \frac{3(k+1)}{2}DJ(w)-(k+3)T(w))(z-w)^{-1}, &
\end{align*}
where $c_k = -\frac{(3k +1)(2k +3)}{k+3}$.
\end{definition}

Vertex algebra $\mathcal{W}^k$ is called the universal Bershadsky-Polyakov vertex algebra of level $k$. For $k \ne -3$ , $\mathcal{W}^k$ has a unique simple quotient which is denoted by $\mathcal{W}_k$. 

Let
\begin{align*}
 T(z) &= \sum _{n \in \mathbb{Z}} L_nz^{-n-2} &\\
 J(z) &=  \sum _{n \in \mathbb{Z}} J_nz^{-n-1}, &\\ 
 G^+(z) &=  \sum _{n \in \mathbb{Z}} G^+_nz^{-n-1}, &\\ 
 G^-(z) &=  \sum _{n \in \mathbb{Z}} G^-_nz^{-n-1}.   &
\end{align*}

The following commutation relations hold:
\begin{align*}
    [J_m, J_n ] &= \frac{2k+3}{3} m\delta _{m+n,0} , \quad [J_m, G^{\pm}_n] = \pm G^{\pm}_{m+n},   &\\ 
    [L_m, J_n] &= - nJ_{m+n}, &\\ 
    [L_m, G^{\pm}_n] &= (\frac{1}{2}m-n+\frac{1}{2})G^{\pm}_{m+n}, &\\  [G^+_m,G^-_n] &= 3(J^2)_{m+n-1} + \frac{3}{2}(k+1)(m-n)J_{m+n-1}  -  (k+3)L_{m+n-1} + &\\ &+ \frac{(k+1)(2k+3)(m-1)m}{2} \delta _{m+n,1}. &
\end{align*}

%definition of highest weight modules (existence follows from Zhu theory)
\subsection{Structure of the Zhu algebra  $A(\mathcal{W}^k)$}
 
By applying results from \cite{KW} we see that for every $(x,y) \in \mathbb{C}^2$ there exists an irreducible representation $L_{x,y}$ of $\mathcal{W}^k$ generated by a highest weight vector $v_{x,y}$ such that
$$ J_{0}v_{x,y} = xv_{x,y}, \quad J_{n}v_{x,y} = 0 \; \text{for} \: n>0, $$
$$ L_{0}v_{x,y} = yv_{x,y}, \quad L_{n}v_{x,y} = 0 \; \text{for} \: n>0, $$
$$ G^{\pm}_{n}v_{x,y} = 0 \; \text{for} \: n\geq 1. $$

\medskip

Let $A_{\omega}(V)$ denote the Zhu algebra associated to the VOA $V$ (cf. \cite{Z}) with the Virasoro vector $\omega$, and let $[v]$ be the image of $v \in V $ under the mapping $V \mapsto A_{\omega}(V)$. 

For the Zhu algebra $A_{\omega}(\mathcal{W}^k)$ it holds that:
\begin{proposition} [\cite{AK-2019}, Proposition 3.2.]
	There exists a homomorphism  $\Phi :  {\Bbb C}[x,y] \rightarrow A_{\omega}(\mathcal{W}^k)$ such that
	$$ \Phi ( x ) = [J], \ \Phi(y) = [T]. $$ 
\end{proposition}

It can be shown that the homomorphism $\Phi :  {\Bbb C}[x,y] \rightarrow A_{\omega}(\mathcal{W}^k)$ is in fact an isomorphism, i.e. that $A_{\omega}(\mathcal{W}^k) \cong {\Bbb C}[x,y]$.

\medskip

If we switch to a shifted Virasoro vector $\overline \omega = \omega + \frac{1}{2} DJ$,
the  vertex algebras $\mathcal W^k$ and $\mathcal W_k$ become ${\Z}_{\ge 0}$--graded with respect to $L(0) =\overline \omega _1$. In this case the Zhu algebras are no longer commutative, and they carry more information about the representation theory.
The Zhu algebra associated to $\mathcal{W}^k$ is then realized as a quotient of another associative algebra, the so-called Smith algebra $R(g)$ (introduced in \cite{S}, see also \cite{DLM-smith}). These algebras were used by T. Arakawa in the paper \cite{A1} to prove rationality of $\mathcal{W}_k(sl_3,f_{\theta})$ for $k= p/2-3$, $p \ge 3$, $p$ odd. 

 We expand the original definition of Smith algebras $R(f)$ by adding a central element.

\begin{definition}\label{smith_type}
	Let $ g(x,y) \in \mathbb{C}[x,y] $ be an arbitrary polynomial. Associative algebra \textbf{ $R(g)$ of Smith type} is generated by $\left\{E,F,X,Y\right\}$  such that $Y$  is a central element and the following relations hold:
	\begin{equation*}
	XE-EX = E, \: XF-FX=-F, \:EF-FE=g(X,Y).
	\end{equation*}
\end{definition}

In fact, Zhu algebra associated to $\mathcal{W}^k$ is a quotient of the Smith-type algebra $R(g)$ for a certain polynomial $ g(x,y) \in \mathbb{C}[x,y] $ .

\begin{proposition}  [\cite{AK-2019}, Proposition 4.2.]
	Zhu algebra $A_{\overline \omega}(\mathcal{W}^k)$ is a quotient of the Smith algebra $R(g)$ for $g(x,y) = -(3x^2 - (2k+3)x - (k+3)y)$.
\end{proposition}

\section{ Vertex algebra  $\mathcal{W}_k$ for $ k +2 \in {\Z}_{\ge 1}$} \label{integer}
In this section we study the representation theory of the Bershadsky-Polyakov algebra $\mathcal W_k$ at positive integer levels.
In \cite{AK-2019}, we parametrized the highest weights of irreducible $\mathcal{W}_k$--modules as zeroes of certain polynomial functions (cf. Proposition \ref{klas-1}), and conjectured that this provides the complete list of irreducible modules for the Bershadsky-Polyakov algebra $\mathcal W_k$ when $k \in \mathbb{Z}$, $k \geq -1$. This conjecture was proved in cases $k=-1$ and $k=0$ in \cite{AK-2019}. In this paper, we will prove this conjecture for $k \in\mathbb{Z}$, $k \geq 1$.

\subsection{Setup}

Let us choose a  new Virasoro field  $$  L (z)  := T(z) + \frac{1}{2} D  J(z).$$ Then $ \overline \omega = \omega + \frac{1}{2} DJ$ is a conformal vector $ \overline \omega_{n+1} =  L(n) $ with central charge $$\overline c_k = -\frac{4(k +1)(2k +3)}{k+3}.$$ The fields $J, G^+, G^-$ have conformal weights $1,1,2$ respectively.
Set $ J(n) = J_n, \, G^+ (n) = G^+ _n, \, G^-  (n) = G^-  _ {n+1}$. We have
 $$  L(z) = \sum_{n \in {\Z} }  L(n) z^{-n-2}, \; G^+ (z) = \sum_{n \in {\Z} }  G^+ (n) z ^{-n-1}, \;  G^- (z) = \sum_{n \in {\Z} } G^- (n) z^{-n-2}. $$
 
 %{\color{red} Svuda treba  promijeniti  $\overline L(n)$  u $L(n)$.}
 
This defines a  $\mathbb{Z} _{\ge 0}$-gradation on $\mathcal{W}^k$.

\medskip

Define $$ \Delta (-J,z) = z^{-J(0)}exp \left ( \sum _{k=1} ^{\infty} (-1)^{k+1}\frac{-J(k)}{kz^k}\right ), $$ and let $$ \sum _{n \in \mathbb{Z}} \psi (a_n)z^{-n-1} = Y (\Delta (-J,z)a, z), $$ for $a \in \mathcal{W}^k$.

The operator $\Delta (h,z )$ associates to every $V$-module $M$ a new structure of an irreducible $V$-module (cf. \cite{Li}). Let us denote this new module (obtained using the mapping $a_n \mapsto \psi (a_n)$) with $\psi (M)$.  As  the $\Delta$-operator acts bijectively on the set of irreducible modules, there exists an inverse  $\psi^{-1}(M)$.

\begin{remark}
{ Operators $\psi^{m}$ are also called  the spectral flow automorphisms of $\mathcal W_k$, see  \cite[Section 2]{AKR-2020} and \cite[Subsection 2.2]{KRZ} for more details.}
\end{remark}

From the definition of $\Delta (-J,z)$ we have that
$$ \psi (J(n)) = J(n)- \frac{2k+3}{3} \delta _{n,0} , \quad \psi ( L(n)) = L(n) - J(n) + \frac{2k+3}{3} \delta _{n,0}, $$
$$ \psi (G^+(n)) = G^+(n-1), \quad \psi (G^-(n)) = G^-(n+1). $$

Let $$ L(x,y)_{\text{top}} = \left\lbrace v \in L(x,y) :  L(0)v = y v \right\rbrace, $$ and denote $$ \hat{x}_i = x+i-1 -\frac{2k+3}{3}, \ \hat{y}_i = y-x-i+1 +\frac{2k+3}{3}.$$ The following was proved in \cite{A1}:
\begin{lemma} [\cite{A1}, Proposition 2.3.] \label{L_new}
	Let $ \dim L(x,y)_{\mbox{top}} = i$. Then it holds that $$ \psi(L(x,y)) \cong L(\hat{x}_i, \hat{y}_i). $$
\end{lemma}

\medskip

\subsection{Necessary condition for $\mathcal{W}_k$--modules}

Starting point in the classification of irreducible $\mathcal{W}_k$--modules for integer levels $k$ is the following formula for singular vectors (cf. \cite{AK-2019}). These generalize a construction of a family of singular vectors by T. Arakawa in \cite{A1}, where he found a similar formula for singular vectors in $\mathcal W^k$ at levels $k= p/2-3$, $p \ge 3$, $p$ odd.

\begin{lemma} [\cite{AK-2019}, Lemma 8.1.] \label{integral-sing}
	Vectors $$G^+(-1)^{n}\mathbbm{1}, \; G^-(-2)^{n}\mathbbm{1}$$ are singular in $\mathcal{W}^k$ for $n = k+2$, where $k \in \mathbb{Z}$, $k \geq -1$.
\end{lemma}

Let $$g(x,y) = -(3x^2 - (2k+3)x - (k+3)y)\in \mathbb{C}[x,y]$$ and define polynomials $h_i(x,y)$, for $i\in \mathbb{N}$ (cf. \cite{A1}) as
\begin{align*} 
h_i(x,y) &= \frac{1}{i}(g(x,y) + g(x+1,y) + ... + g(x+i-1,y)) &\\
&= -i^2+ki-3xi+3i-3x^2-k+2kx+6x+ky+3y-2.&
\end{align*}

The next technical lemma follows immediately from the definiton of $h_i(x,y)$ and $\hat{x}_i, \hat{y}_i$.

\begin{lemma} \label{i+j}
	Assume that $ i+j = k+3 $. Then it holds that  $$ h_i (x,y) = h_j (\hat{x}_i, \hat{y}_i). $$
\end{lemma}

\begin{proof}
 We have
\begin{gather*}
\begin{split} 
	h_{k+3-i} (\hat{x}_i,\hat{y}_i)
	& =(-3x^2-9x+12x+6x-3x + 3xi-3kx-6xi+4kx+2kx-kx) + \\ & + (ky+3y) + (-i^2+2ki+ 
	+6i-k^2-6k-9+k^2+3k-ki+3i^2- \\ &  - 5ki-15i+2k^2 + 12k+18  +3k+9-3i
	-3i^2+4ki+ \\
	& +12i-\tfrac{4}{3}k^2-8k-12-k+2ki-2k-\tfrac{4}{3}k^2+6i-6-
	4k- \\ & -6-ki+k+\tfrac{2}{3}k^2+k-3i+3+2k+1) \\
	& =(-3x^2-3ix+2kx+6x) + (ky+3y) + (-i^2+ki+3i-k-2) \\
	&= h_i(x,y).
	\end{split}
	\end{gather*}
\end{proof}

Define the set 
$$  \mathcal S_k= \left\{  (x,y)  \in {\C} ^ 2 \ \vert \exists i,  \  1 \leq i \leq k+2, \  h_i(x,y) = 0 \right\}.$$

%\marginpar{\footnotesize \color{red} Dodajte i slutnju iz naseg clanka i recite da smo ju u clanku dokazali za $k=-1, 0$.}

In \cite{AK-2019} we proved that in order for $L(x,y)$ to be an irreducible ordinary $ \mathcal W_k$--module, $(x,y)$ needs to belong to the set $\mathcal S_k $.
\begin{proposition}[\cite{AK-2019}] \label{klas-1}
	Let $k \in \mathbb{Z}$, $k \geq -1$. Then we have:
	\begin{itemize}
	\item[(1)] The set of equivalency classes of irreducible ordinary $ \mathcal W_k$--modules   is contained in the set $$ \{ L(x,y) \ \vert \  (x,y) \in  \mathcal S_k \}. $$
	\item[(2)]  Every irreducible   $ \mathcal W_k$--module in the category $\mathcal O$ is an ordinary module.
	\end{itemize}
\end{proposition}

We stated the following conjecture (and proved it for  $k=-1$ and $k= 0$):

\begin{conj} \label{slutnja-1}\cite{AK-2019}
The set   $ \{ L(x,y) \ \vert  \ (x,y) \in  \mathcal S_k \} $	 is the set of all irreducible ordinary $\mathcal W_k$--modules.
\end{conj}

\smallskip

 The proof of Conjecture  \ref{slutnja-1} is reduced to showing that   $L(x,y)$, $(x,y) \in  \mathcal S_k$,  are indeed $\mathcal W_k$--modules.  In what follows we shall prove Conjecture  \ref{slutnja-1}.
 %We still can not prove this in full generality, but we can prove  that $\mathcal W_k$ has $2$-parameter family of irreducible, highest weight modules.
 
 \subsection{Simple current $\mathcal W_k$--modules}

First step in the proof is to construct an infinite family of irreducible $\mathcal W_k$--modules $L(x,y)$  satisfying the conditions $h_1(x,y) = 0$ or $h_{k+2}(x,y) = 0$.
We have the following important lemma:

\begin{lemma} \label{aut-aplication} Assume that $n \in {\Z}_{\ge 0}$.  Define
\bea x_{2n} &=& - n \frac{k+3}{3}, \nonumber \\   x_{2n+1} &=& -n-1 - \frac{( n+2) k}{3}, \nonumber \\
 y_{2n} &=& \frac{n}{3} (3 + 2 k  +  (k+3) n), \nonumber \\
 y_{2n+1} &=&  \frac{n+1}{3}(  n(k+3)  + 2k+3 ) . \nonumber \eea
Then  for each $n \in {\Z}_{\ge 0} $ we have
\begin{itemize}
\item $L(x_n, y_n)$  is irreducible $\mathcal W_k$--module.
\item  $\Psi ^n (\mathcal W_k) \cong L(x_n, y_n)$.
 \item  $h_1 (x_{2n}, y_{2n}) = h_{k+2} (x_{2n+1}, y_{2n+1} ) = 0$.
\end{itemize}
\end{lemma}
%\marginpar{ Dodajte formule za $n <0$}
\begin{proof} 
  
 By direct calculation we have  \bea h_i(x_{2n}, y_{2n}) = 0 &\iff & i \in \{1,  k+2+ n (k+3) \}, \nonumber \\
 h_i(x_{2n+1}, y_{2n+1}) = 0  &\iff &   i \in \{k+2,  2 (k+2) + n (k+3) \}. \nonumber \eea

 We see that $(x_n,y_n)$ is the unique  solution of the following recursive relations:
 \bea 
&& x_{0} = y_{0} = 0, \nonumber \\
&& x_{2n+1} =    x_{2n} - \frac{2k+3}{3}, \ y_{2n+1} = \hat y_{2n} =  y_{2n} -  x_{2n+1},  \nonumber \\
&& x_{2n+2} =  \hat x_{2n+1} = x_{2n+1}  + \frac{k}{3}, \ y_{2n+2} = \hat y_{2n+1} =  y_{2n+1} -  x_{2n+2}.   \nonumber 
\eea
Then for each $n \in {\Z}_{\ge 0}$ we have 
$$L( x_{n}, y_{n}) = \Psi^{n} (L(0,0)) = \Psi^{n} (\mathcal W_k). $$
So $L( x_{n}, y_{n})$ is an  irreducible $\mathcal W_k$--module.

\end{proof}

\begin{lemma} 
Assume that $n \in {\Z}_{<0}$.  Define
\bea x_{2n} &=& - n \frac{k+3}{3}, \nonumber \\   x_{2n-1} &=& 1-n - \frac{( n-2) k}{3}, \nonumber \\
 y_{2n} &=& -\frac{n}{3} (k  - (k+3) n), \nonumber \\
 y_{2n-1} &=&  -\frac{n}{3}( 2k+3 - n(k+3)).  \nonumber \eea
 Then  for each $n \in {\Z}_{< 0} $ we have
\begin{itemize}
\item $L(x_n, y_n)$  is irreducible $\mathcal W_k$--module.
\item  $\Psi ^n (\mathcal W_k) \cong L(x_n, y_n)$.
 \item  $h_{k+2} (x_{2n}, y_{2n}) = h_{1} (x_{2n+1}, y_{2n+1} ) = 0$.
\end{itemize}
\end{lemma}

\begin{proof}

 By direct calculation we have  \bea h_i(x_{2n}, y_{2n}) = 0 &\iff & i \in \{k+2,  1+ n (k+3) \}, \nonumber \\
 h_i(x_{2n-1}, y_{2n-1}) = 0  &\iff &   i \in \{1,  -1-k + n (k+3)  \}. \nonumber \eea

We see that $(x_n,y_n)$ is the unique  solution of the following recursive relations:
 \bea 
&& x_{0} = y_{0} = 0, \nonumber \\
&& x_{2n-1} =    x_{2n} + \frac{2k+3}{3}, \ y_{2n-1} =  y_{2n} +  x_{2n}, \nonumber \\
&& x_{2n-2}  = x_{2n-1}  - \frac{k}{3}, \ y_{2n-2} =   y_{2n-1} +  x_{2n-1}.   \nonumber 
\eea

We have
$$ \Psi^{-1} (L(0,0)) = \Psi^{-1} (\mathcal W_k)= L(\frac{2k+3}{3},0). $$
Then for each $n \in {\Z}_{\le  0}$ we have 
$$L( x_{n}, y_{n}) = \Psi^{n} (L(0,0)) = \Psi^{n} (\mathcal W_k). $$
Hence $L( x_{n}, y_{n})$ is an irreducible $\mathcal W_k$--module.
\end{proof}

 By using     \cite[Theorem 2.15]{Li} (see also  \cite[Proposition 3.1]{AP-2019}) we get:

 \begin{corollary}
In the fusion algebra of $\mathcal W_k$, $L(x_n,y_n)$ are simple-current modules, i.e., the following fusion rules hold:
$$\Psi^{n} (\mathcal W_k) \times  L(x,y) =  L(x_n, y_n) \times L(x,y) = \Psi^{n} (L(x,y)). $$
\end{corollary}

\medskip

\subsection{Modules $L(x,y)$ such that $h_1(x,y) = 0$ or $h_{k+2}(x,y) = 0$}

\begin{theorem} \label{modules-1}
Assume that $h_1(x,y) =0$ or $h_{k+2}(x,y) = 0$. Then $L(x,y)$ is an irreducible ordinary $\mathcal W_k$--module.
\end{theorem}
\begin{proof}
The solution of the equation  $h_1(x,y) = 0$ is 
$$ y= g_1(x) = \frac{-3 x - 2 k x + 3 x^2}{3 + k} \quad (x \in {\C}). $$
So we need to prove that
$L(x,   g_1(x) )$ is an $\mathcal W_k$--module for every $x \in {\C}$.   On the other hand, the Zhu's algebra $A_{\omega} (\mathcal W_k)$ is isomorphic to a certain quotient of ${\C}[x,y]$ by an ideal $I$. 
Since  ${\C}[x,y]$  is Noetherian, the ideal $I$  is finitely generated  by finitely many polynomials, say $P_1, \dots, P_{\ell}$. Hence the highest weights $(x,y)$ of irreducible $\mathcal W_k$--modules are solutions of the equations:
$$P_i (x, y + x/2) = 0, \ \ i= 1, \dots, \ell. $$
It remains to prove that
$$P_i (x, g_1(x) + x/2) = 0, \ \forall x \in {\C}. $$
By Lemma  \ref{aut-aplication}, we have that $L(x, g(x))$ are $\mathcal W_k$--modules for $x=x_{2n} =  - n \frac{k+3}{3}$.
So we have 
$$P_i (x_{2n}, g_1 (x_{2n}) + x_{2n}/2 ) = 0,  \ \ i= 1, \dots, \ell. $$
Since each  $P_i (x, g_1(x) + x/2) $ is a polynomial in one variable and given that it already has infinitely many zeros, it follows that
$P_i (x, g_1(x) + x/2) \equiv  0$. 
This proves that $L(x,   g_1(x) )$  is  $\mathcal W_k$--module, for every $x \in {\C}$. 

Applying Lemma  \ref{i+j}, we get that $L(x,y)$ is a $\mathcal W_k$--module for each solution of the equation $h_{k+2} (x,y) =0$.
\end{proof}

\subsection{ Proof of  Conjecture  \ref{slutnja-1}}

The proof of the conjecture is reduced to the existence of    irreducible  $\mathcal W_k$ modules $L(x,y)$ such that  $h_i(x,y) = 0$ for arbitrary $1 \le i \le k+2$.
 
\begin{lemma} Let $1 \le i \le k+2$. 
There exists an irreducible $\mathcal W_k$--module $L(x^i,y^i)$  with highest weight
$$(x^i,y^i) = \left( \frac{1-i}{3}, \frac{(-1 + i) (-1 + i - k)}{3 (3 + k)} \right) $$ such that
$\dim L(x,y) _{top} = i$.
\end{lemma}
\begin{proof}  Note that
$$ h_j (x^i,y^i) =  (i - j) (-2 + j - k) = 0 \iff j  \in \{ i, k+2\}. $$
By Theorem  \ref{modules-1} we know that $L(x^i,y^i)$ is an $\mathcal W_k$--module. But since $h_i(x^i,y^i)=0$ and $h_{j}(x^i,y^i) \ne 0$ for $j <i$, we conclude that $\dim L(x^i,y^i) _{top} = i$.
The proof follows.
\end{proof}

Now we shall continue as in the previous subsection.  We will apply the automorphism $\Psi$ and construct  new infinite family of $\mathcal W_k$--modules.

\begin{lemma} \label{beskonacno-i} For each $n \in {\Z}_{\ge 0}$,
$\Psi^n (L(x^i,y^i)) = L (x^i_n, y^i_n)$, where
\bea
 x^i _{2n} &=& \frac{1-i}{3} - n \frac{k+3}{3}    \nonumber \\  y^i _{2 n} &=&  \frac{ 1 - 2 i + i^2 + k - i k + 12 n - 3 i n + 10 k n - i k n + 2 k^2 n + 
 9 n^2 + 6 k n^2 + k^2 n^2}{3 (3 + k) }\nonumber  \\
 x^i _{2n+1} &=&  \frac{-5 + 2 i - 3 n - k (2 + n)}{3}   \nonumber \\  y^i _{2 n+1} &=&   \frac{16 - 8 i + i^2 + 12 k - 3 i k + 2 k^2 + 21 n - 3 i n + 16 k n - 
 i k n + 3 k^2 n + 9 n^2 + 6 k n^2 + k^2 n^2}{3 (3 + k)} \nonumber  \eea
\end{lemma}
\begin{proof}
Using direct calculation we get
$$h_{j} (x_{2n}, y_{2n}) = (i - j) (-2 + j - 3 n - k (1 + n)), $$
$$h_{j} (x_{2n+1}, y_{2n+1}) = -(-3 + i + j - k) (-5 + i + j - 2 k - 3 n - k n),$$
which implies that 
$$ \dim L(x_{2n}, y_{2n})_{top} = i, \ \dim   L(x_{2n+1}, y_{2n+1})_{top} = k+3-i. $$

We see that $(x^i_n,y^i_n)$ is the unique  solution of the following recursive relations:
 \bea 
&& x^i_{0} = x^i,  y^i_{0} = y^i, \nonumber \\
 && x^i_{2n+1} =    x^i_{2n}  + (i-1) - \frac{2k+3}{3}, \ y^i_{2n+1} =   y^i_{2n} -  x^i_{2n+1},  \nonumber \\
&& x^i_{2n+2} =   x^i_{2n+1}  +   k+ 2 -i - \frac{2k+3}{3}, \ y_{2n+2} = \  y^i_{2n+1} -  x^i_{2n+2}.   \nonumber 
\eea
This proves that  for each $n \in {\Z}_{\ge 0}$ we have 
$$L( x^i_{n}, y^i_{n}) = \Psi^{n} (L(x^i,y^i)). $$
Since $L(x^i,y^i)$ is a $\mathcal W_k$--module, we have that  $L( x^i_{n}, y^i _{n})$ is an  irreducible $\mathcal W_k$--module. The proof follows.
\end{proof}

\begin{theorem} \label{modules-general}
Assume that $h_i (x,y) =0$ for $1 \le i \le k+2$. Then $L(x,y)$ is an irreducible ordinary $\mathcal W_k$--module.
\end{theorem}
\begin{proof}
By Lemma \ref{beskonacno-i} we have that  $L(x^i_{2n}, y^i_{2n})$ are $\mathcal W_k$--modules for every $n \in {\Z}_{\ge 0}$. 
Since $h_{i} (x^i_{2n}, y^i_{2n}) =0$, we conclude that there is an infinite-family of points $(x,y)$  of the curve $h_{i}(x,y)=0$ such that $L(x,y)$ is a $\mathcal W_k$--module. Applying the same argument as in the proof of Theorem  \ref{modules-1}, we get that $L(x,y)$ is a $\mathcal W_k$--module for every $(x,y)$ such that $h_i(x,y) =0$.
\end{proof}

 Theorem \ref{modules-general} concludes the proof of Conjecture \ref{slutnja-1}.

\section{The duality of  $\mathcal W_1$  and $L_{-5/4} (osp(1 \vert 2))$.}
\label{duality-section}

In this section we construct an embedding of the Bershadsky-Polyakov vertex algebra $\mathcal W_1$ into the tensor product of the affine vertex {super}algebra $L_{k} (osp(1 \vert 2))$ at level $k=-5/4$ and the Clifford vertex superalgebra $F$.
The affine vertex {super}algebra  $V^k(osp(1 \vert 2)) $ and its modules were realized by the first named author in \cite{A-2019}.  Moreover, we will  show that $L_{-5/4} (osp(1 \vert 2))$ is the Kazama-Suzuki dual of  $\mathcal W_1$ by proving that there is an embedding of 
$\mathcal W_1$ into $F_{-1}$, where $F_{-1}$ is the vertex {super}algebra associated to the lattice ${\Z} \sqrt{-1}$.
%Using this result, we will obtain in next section  a realization of the Bershadsky-Polyakov algebra $\mathcal W_1$ and its modules.

\subsection{ Affine vertex superalgebra $V^k(osp(1 \vert 2)) $} \label{sec:realizacija}
Recall that $\g = osp(1,2)$ is the simple complex Lie superalgebra   with basis $\{e,f, h, x, y\}$ such that the even part
$\g^{0} = \mbox{span}_{\C} \{ e, f, h \}$ and the  odd part $\g^{1}  =\mbox{span} _{\C} \{x, y \}$.
The anti-commutation relations are given by
\bea
 && [e,f] = h, \ [h, e] = 2 e, \ [h, f] = - 2 f \nonumber \\
 && [h, x] = x, \ [ e, x] = 0,  \ [f, x] = -y \nonumber \\
 && [h, y] = -y, \ [e, y] =-x \ [f, y] = 0 \nonumber \\
 && \{ x, x\} = 2 e, \ \{x, y\} = h, \ \{ y, y \} = -2 f. \nonumber
\eea
Choose the non-degenerate super-symmetric bilinear form  $ ( \cdot , \cdot )$ on $\g$ such that non-trivial products are given by
$$ (e, f) = (f,e) = 1, \ (h,h) = 2, \ (x, y) =  -(y, x) = 2. $$
Let ${\widehat \g} = {\g} \otimes {\C}[t, t^{-1}] + {\C} K$ be the associated affine Lie superalgebra, and $V^k (\g)$ (resp. $L_k (\g)$) the associated universal (resp. simple) affine vertex superalgebra.
As usual, we identify  $x \in {\g}$ with $x(-1) {\bf 1}$.

The Sugawara conformal vector of $V^k(osp(1 ,2))$ is given by $$ \omega_{\text{sug}} = \frac{1}{2K+3} \left(:ef: + :fe: + \frac{1}{2}:hh: - \frac{1}{2}:xy: + \frac{1}{2}:yx: \right). $$

{ 

 The notion of Ramond--twisted modules is defined as usual in the case of affine vertex superalgebras  \cite{KW-tw} (see also \cite{SRW}). 

Recall also the spectral flow automorphism $\rho$ for $\widehat{osp}(1 \vert 2)$:
 \bea &&   \rho e(n) = e(n-2),  \quad \rho x(n) = x(n-1),   \quad \rho f(n) = f(n+2), \nonumber \\ &&  \rho y(n) = y(n+1)  \quad   \rho h(n) = h(n) -2 \delta_{n,0} K. \quad (n \in {\Z}). \nonumber \eea

%As in the proof  \cite[Proposition 2.1]{A-2007}  (see also \cite[Proposition 3.1]{AP-2019}, \cite{KR-2019}),
As in the case of the automorphism $\Psi $ of the $\mathcal W_k$, one shows  that for any $L_k (osp(1 \vert 2))$--module $(M, Y_M(\cdot, z))$ and $n \in {\Z}$,
$\rho ^ n (M)$ is again a $L_k (osp(1 \vert 2))$--module with vertex operator structure given by
% {\color{blue} ref?}
$$  Y_{\rho ^n (M) } ( \cdot , z))  := Y_{ M } (\Delta ( -n h, z) \cdot , z)). $$
 (see also  \cite[Proposition 2.1]{A-2007} for the proof of similar statement for spectral-flow automorphism of $L_k(sl(2))$, and \cite[Proposition 3.1]{AP-2019} in the case of $\beta-\gamma$ system).

}

\subsection{Embedding of  $\mathcal W_1$ into  $ L_{-5/4} (osp(1 \vert 2)) \otimes F$}

The free field realization of $L_k(osp(1 \vert 2))$ is presented in \cite{A-2019}. In what follows, we will show that the simple Bershadsky-Polyakov algebra $\mathcal W_1$ can be embedded into $ L_{-5/4} (osp(1 \vert 2)) \otimes F$, where $ F $ is the Clifford vertex superalgebra, introduced in Section
\ref{clifford-cijeli}.
 
 \medskip 

Set 
\begin{align*}
    \tau^+ &= : \Psi^+  x:  &\\
	\tau^-&= : \Psi^-  y :. &
\end{align*}

Let $\alpha = :\Psi^+ \Psi^-:$. Then 
$$ \omega_F = \frac{1}{2} :\alpha \alpha: . $$

Denote  by $M_{h ^{\perp}}  (1)$  the Heisenberg vertex algebra generated by $ h ^{\perp} = \alpha - h$.  For $s \in {\C}$, denote by $M_{h ^{\perp}}  (1, s)$   the irreducible  $M_{h ^{\perp} }  (1)$--module on which $h ^{\perp} (0)$ acts as $s  \mbox{Id}$.
Note that
$h^{\perp} (1) h^{\perp} = -\tfrac{3}{2} {\bf 1}$. The Virasoro vector of central charge $c=1$ in $M_{h ^{\perp}}  (1)$  is
$\omega^{\perp} = -\frac{1}{3} h^{\perp} (-1) ^2{\bf 1}$.
 
The following lemma can be proved easily using results from \cite{Li, Li-ext}.
\begin{lemma}  \label{heis-sf} Consider the $M_{h ^{\perp}}  (1)$--module $M_{h ^{\perp}}  (1, s)$ with the vertex operator $Y_s(\cdot, z)$.
Then for every $n \in {\Z}$
$$ (\widetilde{M_{h ^{\perp}}  } (1, s), \widetilde{Y_s} (\cdot, z)):= (M_{h ^{\perp}}  (1, s),
 Y_s(\Delta (-\frac{2 n}{3}h ^{\perp}, z) \cdot, z))$$
 is an irreducible $M_{h ^{\perp}}  (1)$--module isomorphic to $M_{h ^{\perp}}  (1, n+ s)$.
\end{lemma}

\begin{theorem}  \label{real-1}
\begin{itemize}
\item[(i)] There is a non-trivial homomorphism of vertex algebras $ \Phi : \mathcal W^1 \rightarrow L_{-5/4} (osp(1 \vert 2)) \otimes F$ uniquely determined by
\bea
G^+ &=& 2 \tau^+ =  2 : \Psi^+  x:  \nonumber \\
G^{-} & =& 2 \tau ^- =  2  : \Psi^-  y : \nonumber \\
J &=& \frac{5}{3} \alpha - \frac{2}{3} h \nonumber \\
 T &=& \omega_{sug}  + \omega_F -\omega^{\perp}. \nonumber \eea

\item[(ii)]  $\mbox{Im} (\Phi) $ is isomorphic to the simple vertex algebra $\mathcal W_1$.

\item[(iii)]
  As a  $ {\mathcal W}_1 \otimes  M_{h ^{\perp}}  (1)$-module:
\bea   L_{-5/4} (osp(1 \vert 2)) \otimes F & \cong&  \bigoplus_{n\in {\Z} } 
 \Psi^{-n} (   {\mathcal W}_1)  \otimes   M_{h ^{\perp}}   (1, n)  \nonumber\\
 &=&  \bigoplus_{n\in {\Z} } 
L( x_{n}, y_{n} ) \otimes   M_{h ^{\perp}}   (1,   - n ). \nonumber
 \eea
 \item[(iv)] We have:
 $$ \mbox{Com} (M_{h^{\perp}}(1), L_{-5/4} (osp(1 \vert 2)) \otimes F) \cong   {\mathcal W}_1. $$
 \end{itemize}

\end{theorem}

\begin{proof}
Let $k=1$.
Using  the formula $$(\Psi^+_{-1}x)_{n} = \sum_{i=0}^{\infty}\left(\Psi^+_{-1-i}x(n+i) - x(n-i-1)\Psi^+_{i}\right) $$  we obtain
\begin{align*}
    \tau^+_2\tau^- &= -2k'\mathbbm{1}, &\\
		\tau^+_1\tau^-&= -2k'\alpha - h, &\\
			\tau^+_0\tau^-&= -\alpha h - :xy: - 2k' :D \Psi^+\Psi^-:. &
\end{align*}

%Letting  $$J=\lambda(2K\alpha + h),$$ from the formulas $J_0\tau^+=  \tau^+$,  $J_1J= \frac{2k+3}{3}$ 

From the above formulas and OPE relations for $\mathcal W^k$ it follows that for $k=1$, $k'$ needs to be equal to $ -5/4$. For level $k'=-5/4$  there exists a conformal embedding of $L_{k'}(sl(2))$ into $L_{k'}(osp(1 \vert 2))$ (cf. \cite{AKMPP-2020}), that is, $$ \omega_{\text{sug}} = \omega_{sl(2)}, $$ where  $$ \omega_{sl(2)} = \frac{1}{2(k'+2)} \left(:ef: + :fe: + \frac{1}{2}:hh: \right). $$
We have
\begin{align*}
G^+ _2 G^- & = 4 \tau^+_2 \tau^-= 10 \mathbbm{1} = (k+1) (2k+3) \mathbbm{1}, &\\
G^+ _1 G^- &=   4 \tau^+_1 \tau^- =  - 4 ( - \frac{5}{2} \alpha + h) = 10 \alpha - 4 h = 6 J= 3 (k+1)J, &\\
G^+ _0 G^-  &=  4 \tau^+_0 \tau^-=- 4  : h \alpha : - 4 : x y: + 5 :\alpha \alpha: + 5 D \alpha. & \end{align*}

\smallskip

Using the realization in \cite{A-2019} and the fact that there is a conformal embedding of $L_{-5/4}(sl(2))$ into $L_{-5/4}(osp(1 \vert 2))$, it follows that
$$\omega_{sug} =  \frac{1}{2} ( : x y : -  y x :)  =  : x y: -\frac{1}{2} D h. $$
This implies that
\bea  T  &=&   : x y: -\frac{1}{2} D h + \frac{1}{2} : \alpha \alpha: + \frac{1}{3}   (: \alpha \alpha: + : h h : - 2 :\alpha h:)  \nonumber \\
&=&  : x y:  -\frac{1}{2} D h   + \frac{5}{6} : \alpha \alpha:  + \frac{1}{3} : h h: - \frac{2}{3} : \alpha h : \nonumber 
\eea
Hence
\bea
G^+ _0 G^-  &=&  - 4  : h \alpha : - 4 : x y: + 5 :\alpha \alpha: + 5 D \alpha \nonumber \\
& =& 3 : J ^ 2 : + 3 : D J: - 4 T \nonumber \\
&=& 3 : J ^ 2 : + \frac{3(k+1)} {2} : D J: - (k+3) T. \nonumber
\eea
This proves assertion (i).

\medskip

Let us prove that   $\overline{\mathcal W}_1= \mbox{Im} (\Phi) $ is simple.   

Let $\overline{W} = \mbox{Ker} _{   L_{-5/4} (osp(1, 2)) \otimes F} h ^{\perp} (0)$. It is clear that $\overline{W}$ is a simple vertex algebra which contains
$ \overline{\mathcal W}_1 \otimes  M_{h ^{\perp}}  (1)$.

The simplicity of   $ \overline{\mathcal W}_1$ follows from the following claim:
\begin{itemize}
\item[\bf Claim 1]  $\overline W$ is generated by $ G^+, G^-, J, T,  h ^{\perp}$.
\end{itemize}

 For completeness, we shall include a proof of the Claim 1.

Let $U$ be the vertex subalgebra of $\overline W$ generated by $ \{G^+, G^-, J, T,  h ^{\perp} \}$. Then clearly $U \cong \overline{\mathcal W}_1 \otimes  M_{h ^{\perp}}  (1)$.
 
Let $ U^{(n)} $ be the $U$--module obtained by the simple current construction
$$ (U^{(n)}, Y^{(n)} (\cdot, z)): = (U , Y(\Delta(n \alpha,) \cdot, z)). $$

 Note next that
$ \alpha  = J- \frac{2}{3} h^{\perp}$, which implies that 
\bea
 \label{delta-alpha}  \Delta(\alpha, z) = \Delta( -\frac{2}{3} h^{\perp}, z) \Delta(J, z).
   \eea

 \begin{itemize}
\item For a $ \overline{\mathcal W}_1$--module $W$, 
by applying operator $\Delta(n J, z)$  we get module $\Psi^{-n}(W)$.
\item  Using Lemma \ref{heis-sf} we see that  by applying the  operator $ \Delta( -\frac{2  n }{3} h^{\perp}, z)$ on 
 $ M_{h^{\perp}}(1)$,   we get module $ M_{h^{\perp}}(1)$--module
$ M_{h^{\perp}} (1,  n )$.
\end{itemize}
We get
\bea 
U^{(n)}= \Psi^{-n}(  \overline{\mathcal W}_1) \otimes   M_{h ^{\perp}}  (1,n ).
\label{dec-iii}
\eea

 Since by the boson-fermion correspondence $F$ is isomorphic to the lattice vertex {super}algebra $V_{\Z\alpha}$, we get that
$ U^{(n)}$ is realized inside of   $ L_{-5/4} (osp(1 \vert 2)) \otimes F$:
$$ U^{(n)} \cong  U. e^{n\alpha}. $$
Note that $ h^{\perp} (0) \equiv n \mbox{Id} $ on $U^{(n)}$.
 Using H. Li construction from  \cite{Li-ext} (see also \cite{K1})  we get that
\bea  \mathcal U = \bigoplus U^{(n)} \label{dec-u1} \eea
is a vertex subalgebra of $ L_{-5/4} (osp(1 \vert 2)) \otimes F$. But one shows that $ \mathcal U$ contains all generators of $ L_{-5/4} (osp(1 \vert 2)) \otimes F$, so 
\bea \mathcal U =   L_{-5/4} (osp(1 \vert 2)) \otimes F. \label{dec-u2}\eea
This proves that $  U= \overline W \cong   \overline{\mathcal W}_1 \otimes  M_{h ^{\perp}}  (1) $. Since $ \overline W$ is simple, we have that  $\overline{\mathcal W}_1$ is simple, and therefore isomorphic to $\mathcal W_1$.  This proves Claim 1.

{ The decomposition (iii) follows from relations  (\ref{dec-iii})-(\ref{dec-u2}).
The assertion (iv) follows directly from (iii).}

\end{proof}

The proof  of the following result is  similar to the one given in \cite[Theorem 6.2]{A-2007} and \cite[Theorem 5.1]{AP-2019}.
 
 \begin{theorem}  \label{ired-general} 
 Assume that $ N$ (resp. $N^{tw}$) is an irreducible, untwisted (resp. Ramond twisted)   $L_{-5/4} (osp(1 \vert 2))$--module  such that $h(0)$ acts semisimply on $ N$ and $N^{tw}$:
 \bea  N &=& \bigoplus_{s \in {\Z} + \Delta}   N^s, \quad h (0) \vert N^s \equiv s \mbox{Id} \quad (\Delta \in {\C}), \nonumber \\
 N^{tw} &=& \bigoplus_{s \in {\Z} + \Delta'}   (N^{tw} )^s, \quad h (0) \vert (N^{tw} )^s \equiv s \mbox{Id} \quad (\Delta' \in {\C}), \nonumber
 \eea
 Then $  N \otimes F$ and $N^{tw} \otimes M_F^{tw}$ are completely reducible $\mathcal W_1$--modules:
 $$  N \otimes F = \bigoplus_{s \in {\Z}}  \mathcal L_s(N),  \quad  \mathcal L_s (N) = \{ v \in     N \otimes F  \ \vert  h^{\perp}(0) v = ( s + \Delta) v \},$$
 $$  N^{tw} \otimes M_F^{tw}  = \bigoplus_{s \in {\Z}}  \mathcal L_s(N^{tw}),  \quad  \mathcal L_s (N^{tw}) = \{ v \in     N^{tw} \otimes M_F^{tw}  \ \vert  h^{\perp}(0) v = ( s + \Delta') v \},$$ 
 and each $ \mathcal L_s(N) $  and  $ \mathcal L_s(N^{tw} ) $ are irreducible $\mathcal W_1$--modules.
 \end{theorem}

\subsection{Embedding of   $L_{-5/4} (osp(1 \vert 2))$ into $\mathcal W_1 \otimes F_{-1}$}
In this section we consider the tensor product of $\mathcal W_1$ with the lattice vertex {super}algebra $F_{-1}$ (defined in Section \ref{F_{-1}}). We will show that the simple affine vertex {super}algebra $L_{-5/4} (osp(1 \vert 2))$ can be realized as a subalgebra of $\mathcal W_1 \otimes F_{-1}$.

Let  $\overline h =  J + \varphi$. Then  for $n \ge 0$ we have $\overline h(n) \overline h = \frac{2}{3} \delta_{n,1} {\bf 1}$.  
Let $M_{\overline h} (1)$ be the Heisenberg vertex algebra generated by $\overline h$, and  $M_{\overline h} (1, s)$ the irreducible  highest weight $M_{\overline h} (1)$--module on which 
$\overline h(0) \equiv s \mbox{Id}$.

\begin{theorem}\label{duality}
\item[(i)]There is a homomorphism of vertex algebras $\Phi^{inv}  :  L_{-5/4} (osp(1 \vert 2)) \rightarrow  \mathcal W_1 \otimes F_{-1} $ uniquely determined by
\bea
 x  &=&    \frac{1}{2}  :  G ^+ e^{\varphi}:  \nonumber \\
 y  & =&     -\frac{1}{2}   :  G ^- e^{-\varphi}:  \nonumber \\
 e &=&         \frac{1}{  8}  :  (G ^+ ) ^2  e^{2 \varphi}:   \nonumber \\
 f & =&   - \frac{1}{  8}   :  (G ^- ) ^2 e^{- 2 \varphi}:  \nonumber \\
h  &=&  -\frac{3}{2} J  - \frac{5}{2} \varphi. \nonumber \eea
\item[(ii)]  We have:
 $$ \mathcal W_1 \otimes F_{-1}  \cong \bigoplus_{n \in {\Z}} \rho^n (L_{-5/4} (osp(1 \vert 2)) \otimes M_{\overline h} (1, -n). $$
\item[(iii)]  
$$  L_{-5/4} (osp(1 \vert 2)) = \mbox{Com}( M_{\overline h} (1),  \mathcal W_1 \otimes F_{-1} ). $$
\end{theorem}
\begin{proof}
 
We need to show that the following relations hold for $n \geq 0$:
\begin{align*}
h(n)x &= \delta_{n,0}x, \quad e(n)x=0, \quad x(n)f=\delta_{n,0}y, &\\
h(n)y &= -\delta_{n,0}y, \quad e(n)y=-\delta_{n,0}x, \quad f(n)y=0, &\\
x(n)x &= 2\delta_{n,0}e, \quad y(n)y=-2\delta_{n,0}f, \quad x(0)y=h, \quad x(1)y=-\frac{5}{2}\mathbbm{1}, &\\
h(0)e &= 2e, \quad h(0)f=-2f, \quad e(0)f=h.
\end{align*}
Direct calculation shows that
 \bea
  x (1)  y  &=&  -\frac{1}{4} G ^+ _2 G = -\frac{(k+1) (2k+3)} {4}    =  - \frac{5}{4} (x, y)  {\bf 1}= -\frac{5}{2} {\bf 1}. \nonumber  \\
    x (0)   y  &=&   -\frac{1}{4}  ( \varphi(-1) G ^+ _2  G  +  G ^+ _1  G) =  -\frac{(k+1) (2k+3)}{4}   \varphi - \frac{  3 (k+1)}{4} J \nonumber  \\
  &=&  -\frac{5}{ 2}   \varphi - \frac{3}{2}  J  = h. \nonumber  \\
  x(0) x &=& {  2} e  \quad    y(0) y  =  -{ 2}  f  \nonumber  \\
  x(0) e & =& 0,  \quad y(0) f = 0 \nonumber \\
  && ( \mbox{note that above we used relations}  \  (G^{\pm} ) ^3 =0 ) \nonumber \\
   h (0) x  &=& x, \quad h(0) y = - y \nonumber \\
   h(1) h &=&  \frac{9}{4} J_1 J  + \frac{25}{4} \varphi (1) \varphi  =  ( \frac{9}{4} \frac{2k+3}{3}  -   \frac{25}{4} ) = - { \frac{5}{2} }  {\bf 1} = - \frac{5}{4} (h,h) {\bf 1}
.  \nonumber 
 \eea

\begin{align*}
h(0)e &= -\frac{3}{{ 16} }J_0(G^+)^2e^{2\varphi}- \frac{5}{{ 16} }(G^+)^2\varphi(0)e^{2\varphi} &\\   
&= -\frac{3}{8}(G^+)^2e^{2\varphi}+ \frac{5}{8}(G^+)^2e^{2\varphi}= \frac{1}{4}(G^+)^2e^{2\varphi} = 2e, &\\
h(0)f &= \frac{3}{{ 16} }J_0(G^-)^2e^{-2\varphi}+ \frac{5}{{  16} }(G^-)^2\varphi(0)e^{-2\varphi} &\\   
&= -\frac{3}{8}(G^-)^2e^{-2\varphi}+ \frac{5}{8}(G^-)^2e^{-2\varphi}= \frac{1}{4}(G^-)^2e^{-2\varphi} = -2f, 
\end{align*}
  We use the formula (cf. \cite[Section 8.1]{AK-2019})
$$ G^+ _2(  G ^- )^n =  2n  (k - (n-2))  (k - (n-2)  + n/2) ( G^-)^{n-1},   $$
	\begin{equation*}
	G^+_1(G^-)^{n} = 3n(k-(n-2))J_{-1}(G^-)^{n-1} +  n(n-1)(k-(n-2))G^-_{-2}(G^-)^{n-2}.
	\end{equation*}
which implies that for $n =2$ and $k=1$ we get
$$ G^+ _ 2  (G^-)^2  = ( 4 *  1 * 2)    G^- = 8 G^- .$$
%{\color{red} Treba jos dodati formulu za $G^+_1  (  G ^- )^n  $ s tocnom referencom.}

We have
\begin{align*}
 x(0)f &= -\frac{1}{16 }G^+_2(G^-)^2e^{-\varphi}= -\frac{1}{16}8G^-e^{-\varphi} = -\frac{1}{2} G^-e^{-\varphi} = y, &\\
e(0)y &= -\frac{1}{ 16}(2G^+_{-1}G^+_2G^-e^{\varphi} + 2G^+_{0}G^+_1G^-e^{\varphi} ) = -\frac{1}{ 16}(2(k+1)(2k+3)G^+e^{\varphi} + 6(k+1)G^+_{0}J_{-1}e^{\varphi} ) &\\
&= -\frac{1}{ 16} (20G^+e^{\varphi} -12G^+e^{\varphi} ) = - \frac{1}{2} G^+e^{\varphi} = -x, &\\
e(0)f &= -\frac{1}{ 64}(2G^+_{1}G^+_2(G^-)^2 + (G^+_2)^2(G^-)^2 {\ 2} \varphi) = -\frac{1}{ 64}(16G^+_1G^- +{  16} G^+_{2}G^-\varphi ) &\\
&= -\frac{1}{ 64}(16\cdot 3(k+1)J + {  16} (k+1)(2k+3)\varphi ) = -\frac{3}{2} J-\frac{5}{2} \varphi = h.
\end{align*}

\medskip

 %Define next $\overline h =  J + \varphi$. Then  for $n \ge 0$ we have $\overline h(n) \overline h = \frac{2}{3} \delta_{n,1} {\bf 1}$.  
%Let $M_{\overline h} (1)$ be the Heisenberg vertex algebra generated by $\overline h$. 
{
The operator $\overline h(0)$ acts semi--simply on  $\mathcal W_1 \otimes F_{-1}$ and we have the following decomposition:
\bea && \mathcal W_1 \otimes F_{-1} =\bigoplus _{n \in {\Z}} Z^{(n)}, \quad  Z^{(n)} = \{ v \in  \mathcal W_1 \otimes F_{-1} \vert \ \overline h(0) v = n v \}. \label{dec-z1} \eea
We have that  $Z^{(0)}$ is a simple vertex  superalgebra and each $Z^{(n)}$  is a simple $Z^{(0)}$--module.
}
The rest of the proof  follows from the following claim 
\begin{itemize}
\item[\bf Claim 2.] $ Z^{(0)}= \mbox{Ker} _{ \mathcal W_1 \otimes F_{-1}} \overline h(0)$ is  generated by $\{ x,y, e,f, h, \overline h\}$.
\end{itemize}

The proof of the Claim 2 is completely analogous to the proof of Claim 1. These arguments are also similar to those in \cite[Theorem 6.2]{A-2007}, \cite[Proposition 5.4]{AKMPP-2018}, \cite[Section 5]{CKLR} in a slightly different setting.

{ 
The Claim 2  implies that
$$ Z^{(0)}  \cong L_{-5/4} (osp(1 \vert 2)) \otimes M_{\overline h} (1). $$

As in the proof of Theorem   \ref{real-1}, using formula
$$ \Delta(n \varphi, z) = \Delta(-n h , z)  \Delta( -\frac{3 n}{2} \overline h , z)  ,$$
 we get 
\bea Z^{(-n)} = \rho^{n} (L_{-5/4} (osp(1 \vert 2)) ) \otimes M_{\overline h} (1, -n). \label{dec-z2} \eea

The decomposition (ii) follows now from  relations (\ref{dec-z1}) and (\ref{dec-z2}). The assertion (iii) is a direct consequence of (ii).

This concludes the proof of the Theorem.}
\end{proof}

% \section{Explicit realization of $\mathcal W_k$-modules for $k=1$}
  \section{    From   relaxed  $L_{-5/4}(osp(1\vert 2 ))$--modules to ordinary $\mathcal W_1$-modules }

 The free-field realization of  the affine vertex superalgebra $L_{k'} (osp(1 \vert 2))$ and its irreducible modules was obtained in  \cite{A-2019}. Specifically for $k'=-5/4$, it holds that $L_{k'}(osp(1 \vert 2))$ can be realized on the  vertex {super}algebra
%\marginpar{\footnotesize \color{red} Mislim da smo rekli da cemo $F$ iz mog clanka oznacavati s $F^{1/2}$ jer je $c=1/2$, kako bi bio razlicit od ovog $F$.}
%
 $F^{1/2} \otimes \Pi^{1/2}(0)$, where $\Pi^{1/2}(0)$ is a certain lattice type vertex algebra (cf. \cite[Theorem 11.3]{A-2019}). All irreducible $L_{k'}(osp(1 \vert 2))$-modules can be constructed using this realization.
 
In this section we will construct an explicit realization of irreducible $\mathcal W_1$-modules, using the fact that the Bershadsky-Polyakov algebra $\mathcal W_1$ can be embedded into $L_{k'}(osp(1 \vert 2)) \otimes F$, where $F$ is a Clifford vertex superalgebra (cf. Theorem  \ref{real-1}),

\subsection{Relaxed $L_{k'} (osp(1 \vert 2))$--modules} 
 The vertex algebra $\Pi^{1/2}(0)$ was introduced in \cite{A-2019}, where $$\Pi^{1/2}(0) =M(1) \otimes {\C}[\mathbb{Z}\frac{c}{2}]. $$ It is closely related to the lattice type vertex algebra $\Pi(0)$ from \cite{BDT}. Here $c:=\frac{2}{k}(\mu - \nu)$, where $\mu, \nu$ satisfy $\langle \mu, \mu \rangle =- \langle \nu, \nu \rangle = \frac{k}{2}$,  $\langle \mu, \nu \rangle = \langle \nu, \mu \rangle = 0$. It is easy to see that $g= \exp (\pi i \mu (0))$ is an automorphism of order two for $\Pi^{1/2}(0)$.
 
 We will need the following fact about $\Pi^{1/2}(0)$-modules from \cite{A-2019}:
\begin{proposition} \label{tw-pi0} \cite[Proposition 4.1]{A-2019}
Let $\lambda \in \mathbb{C}$ and $g= \exp (\pi i \mu (0))$. Then  $\Pi^{1/2}_{-1}(\lambda):= \Pi^{1/2}(0) e^{-\mu + \lambda c}$ is an irreducible $g$-twisted $\Pi^{1/2}(0)$-module.
\end{proposition}

Using the realization of $\mathcal W_1$, we have that:

\begin{lemma} \label{tw-w1}
Assume that $U^{tw}$ is any $g$--twisted $\Pi^{1/2}(0)$-module. Then $F \otimes M^{\pm}_{F^{1/2}}  \otimes U^{tw}$ and $M_F^{tw} \otimes F^{1/2} \otimes U^{tw}$ are $\mathcal W_1$-modules.
\end{lemma}
\begin{proof}
In \cite[Corollary 13.1]{A-2019} it was proved that if  $U^{tw}$ is any $g$--twisted $\Pi^{1/2}(0)$-module, then $M^{\pm}_{F^{1/2}} \otimes U^{tw}$ is an untwisted $L_{-5/4}(osp(1 , 2))$-module, and $F^{1/2} \otimes U^{tw}$ is a Ramond twisted $L_{-5/4}(osp(1 , 2))$-module. The claim now follows from the realization of the vertex algebra $\mathcal W_1$ in Theorem \ref{real-1}.
\end{proof}

We will consider the following $F^{1/2} \otimes \Pi(0)$--modules:
\begin{itemize}
\item  $\sigma_{F^{1/2}}  \otimes g$--twisted module $\mathcal F_{\lambda} := M^{\pm}_{F^{1/2}}  \otimes \Pi^{1/2}_{-1}(\lambda)$.
\item  $g = 1 \otimes g$--twisted module $\mathcal E_{\lambda} := F^{1/2}  \otimes \Pi^{1/2}_{-1}(\lambda)$.
\end{itemize}
First we recall the result from \cite{A-2019}.
\begin{proposition}  \cite[Theorem 13.2]{A-2019} \label{F_irred}
   $\mathcal F_{\lambda}$ is an untwisted, relaxed $L_{-5/4}(osp(1 , 2))$--module.  $\mathcal F_{\lambda}$ is irreducible if and only of $\lambda \notin \frac{1}{8} + \tfrac{1}{2} {\Z}. $ 
\end{proposition}
%\marginpar{Provjerite u \cite{SRW} kako oni zovu   module u ovim propozicijama}
 Using irreducibility of relaxed $L_{-5/4}(sl(2))$--modules we get:
\begin{proposition} \label{ired-osp-2}We have:
\begin{itemize}
\item[(1)] $\mathcal E_{\lambda}$ is a Ramond twisted  $L_{-5/4}(osp(1, 2))$--module. 
\item[(2)] $\mathcal E_{\lambda} = \mathcal E_{\lambda} ^0 \oplus \mathcal E_{\lambda} ^1$, where  as $L_{-5/4}(sl(2))$--modules
$$ \mathcal E_{\lambda} ^0 =  L^{Vir} (\tfrac{1}{2}, 0) \otimes   \Pi_{-1}(\lambda) \bigoplus   L^{Vir} (\tfrac{1}{2},\tfrac{1}{2}) \otimes   \Pi_{-1}(\lambda+\tfrac{1}{2}) $$
$$ \mathcal E_{\lambda} ^1 =  L^{Vir} (\tfrac{1}{2}, 0) \otimes   \Pi_{-1}(\lambda+\tfrac{1}{2}) \bigoplus   L^{Vir} (\tfrac{1}{2},\tfrac{1}{2}) \otimes   \Pi_{-1}(\lambda) $$
\item[(3)] $\mathcal E_{\lambda} ^{0}$ (resp.  $\mathcal E_{\lambda} ^{1}$) is  irreducible  Ramond twisted  $L_{-5/4}(osp(1, 2))$--modules if $\lambda  \notin {\Z} \cup   (-\tfrac{1}{4} + {\Z})$ (resp.  $\lambda  + \tfrac{1}{2}    \notin {\Z} \cup   (-\tfrac{1}{4} + {\Z})$).
 \end{itemize}
\end{proposition}
\begin{proof}
 Using Proposition \ref{tw-pi0} and Lemma \ref{tw-w1} we see that $\mathcal E_{\lambda}$ is Ramond twisted $L_{-5/4}(osp(1 \vert 2))$--module. This proves (1). Since as a module for the Virasoro vertex algebra $L^{Vir}(\tfrac{1}{2},0):$
 $$ F^{1/2} = L^{Vir} (\tfrac{1}{2}, 0) \oplus L^{Vir} (\tfrac{1}{2}, \tfrac{1}{2}), $$
 and as a $\Pi(0)$--module:
 $$ \Pi^{1/2} _{-1}(\lambda) =   \Pi_{-1}(\lambda) \oplus \Pi_{-1}(\lambda + \tfrac{1}{2}), $$
 we easily get the decomposition in  (2). Using the irreducibility results from \cite{A-2019} and \cite{KR-2019} we get that as $L_{-5/4}(sl(2))$--modules:
 \begin{itemize}
 \item  $L^{Vir} (\tfrac{1}{2}, 0) \otimes   \Pi_{-1}(\lambda)$ is irreducible iff $\lambda \notin   \frac{-1 \pm 1}{8} + {\Z}$,
  \item  $L^{Vir} (\tfrac{1}{2}, 0) \otimes   \Pi_{-1}(\lambda+ \tfrac{1}{2} )$ is irreducible iff $\lambda +\tfrac{1}{2} \in  \frac{-1 \pm 1}{8} + {\Z}$, 
\item  $L^{Vir} (\tfrac{1}{2}, \tfrac{1}{2}) \otimes   \Pi_{-1}(\lambda)$ is irreducible iff $\lambda  \notin   \frac{-1 \pm 5}{8} + {\Z}$,
  \item  $L^{Vir} (\tfrac{1}{2}, \tfrac{1}{2}) \otimes   \Pi_{-1}(\lambda+ \tfrac{1}{2} )$ is irreducible iff $\lambda  +\tfrac{1}{2}  \notin   \frac{-1 \pm 5}{8} + {\Z}$.
 \end{itemize}
 One easily see that all the modules appearing above are irreducible as  $L_{-5/4}(sl(2))$--modules  if and only if $\lambda \notin \tfrac{1}{4}{\Z}$. Using the decomposition
 $L_{-5/4}(osp(1 \vert 2)) = L_{-5/4}(sl(2)) \oplus  L_{-5/4}(\omega_1)$, we easily see  that  as  (Ramond twisted)   $L_{-5/4}(osp(1 \vert 2))$--modules 
 \begin{itemize}
 \item $\mathcal E_{\lambda} ^0$ is irreducible  iff $\lambda \notin
   {\Z}  \cup  (-\tfrac{1}{4}+ {\Z})$
 \item $\mathcal E_{\lambda} ^1$  is  irreducible  iff $\lambda + \tfrac{1}{2} \notin   {\Z}  \cup  (-\tfrac{1}{4}+ {\Z})$.
 \end{itemize}
 The proof follows. \end{proof}
 
\subsection{Explicit realization of $\mathcal{W}_1$--modules} From Theorem \ref{modules-general} it follows that the set $$ \{ L(x,y) \ \vert \  (x,y) \in {\C}^2, \ h_i(x,y)=0, \ 1 \le i \le 3  \} $$ is the set of all irreducible  ordinary $\mathcal W_1$-modules. Now we will construct explicit realizations of these modules, using results from the previous subsection.

\begin{lemma} \label{hw2}
The irreducible highest weight $\mathcal{W}_1$ modules
	$$T_{(2)} :=  \{ L(x, y) \ \vert \ (x,y) \in {\C}^2, \ h_2(x,y)=0 \}$$  are realized as irreducible quotients of $$ U_{(2)} (\lambda) =\mathcal{W}_1.E^{\lambda}_{2}, \quad \lambda \in \mathbb{C},$$ where $E^{\lambda}_{2}= \mathbbm{1}_F \otimes \mathbbm{1}^{tw}_{F^{1/2}} \otimes e^{-\mu + \lambda c} $  are highest weight vectors for $\mathcal{W}_1$ of highest weight
	\bea \label{par-weights} (x_{\lambda} ,y_{\lambda}):=  (  -\frac{2}{3}(-k+ 2\lambda),  -\frac{1}{4}+  \frac{1}{3}(-k+ 2\lambda)^2 +  \frac{1}{3}(-k+ 2\lambda)).  \eea
\end{lemma}
\begin{proof}
Consider the $\sigma_{F^{1/2}}  \otimes g$-twisted $F \otimes F^{1/2} \otimes \Pi^{1/2}(0)$-module $\mathcal{F}_{(2)}(\lambda):= F \otimes M^{\pm}_{F^{1/2}}  \otimes \Pi^{1/2}_{-1}(\lambda)$. Then $\mathcal{F}_{(2)}(\lambda)$ is an untwisted $\mathcal W_1$-module.  It holds that (cf. \cite{A-2019})
\begin{align*}
    h(n)E^{\lambda}_{2} &= \delta_{n,0} (-k+ 2\lambda)E^{\lambda}_{2}, &\\
    L_{sug}(n)E^{\lambda}_{2} &= -\frac{1}{4}\delta_{n,0} E^{\lambda}_{2}, \quad  n \in {\Bbb Z}_{\ge 0}. &
\end{align*}
We have
\begin{align*}
    J(0)E^{\lambda}_{2} &=  - \frac{2}{3}(-k+ 2\lambda)E^{\lambda}_{2}, &\\
    L(0)E^{\lambda}_{2} &= \left(-\frac{1}{4}+  \frac{1}{3}(-k+ 2\lambda)^2 +  \frac{1}{3}(-k+ 2\lambda) \right) E^{\lambda}_{2} & \\
    & = \frac{1}{12} (2 (-k + 2\lambda) + 3) (2 (-k + 2\lambda) -1)E^{\lambda}_{2}.& 
\end{align*}
Set $x_{\lambda}:=  -\frac{2}{3}(-k+ 2\lambda)$ and $y_{\lambda}:=   \frac{1}{12} (2 (-k + 2\lambda) + 3) (2 (-k + 2\lambda) -1) $,
%-\frac{1}{4}+  \frac{1}{3}(-k+ 2\lambda)^2 +  \frac{1}{3}(-k+ 2\lambda)$,
%
 so that  $$J(0)E^{\lambda}_{2}= x_{\lambda}E^{\lambda}_{2}, \quad L(0)E^{\lambda}_{2} =y_{\lambda} E^{\lambda}_{2}.$$ Since $y=\frac{3}{4}x^2 - \frac{1}{2}x - \frac{1}{4}$, the pair $(x_{\lambda},y_{\lambda}) \in \mathbb{C}^2$ satisfies the relation  $$h_2(x,y)= -3x^2+2x+1+4y=0.$$

Hence $\mathcal{W}_1$ has a family of highest weight modules $U_{(2)} (\lambda)$, $\lambda \in \mathbb{C}$, with highest weights $ (x, y)$. In particular, their irreducible quotients $L(x, y)$ are also modules for $\mathcal{W}_1$.
\end{proof}

We have the following irreducibility result:
\begin{proposition}
Assume that $\lambda \notin \frac{1}{8} + \tfrac{1}{2} {\Z}. $  Then  $\mathcal F_{\lambda} \otimes F$ is a completely reducible $\mathcal W_1 \otimes M_{h^{\perp}}(1)$--module:
\bea
\mathcal F_{\lambda} \otimes F  \cong  \bigoplus _{n \in {\Z} } \Psi^{-n} (L(x_{\lambda}, y_{\lambda})) \otimes  M_{h^{\perp}}(1, \Delta+ n) \label{dec-lambda}
\eea
where $\Delta= k - 2 \lambda$ and weights $(x_{\lambda}, y_{\lambda})$ are given by  (\ref{par-weights}).  
In particular, $U_{(2)} (\lambda)$ is an irreducible $\mathcal W_1$-module and it holds that $$ U_{(2)} (\lambda) = L (x_{\lambda} ,y_{\lambda}).$$
\end{proposition}

 \begin{proof}
Since $\mathcal F_{\lambda}$ is an irreducible $L_{-5/4}(osp(1 , 2))$--module for $\lambda \notin \frac{1}{8} + \tfrac{1}{2} {\Z}. $  (cf. Proposition \ref{F_irred}), applying Theorem \ref{ired-general} we see that $\mathcal F_{\lambda} \otimes F$ is a completely reducible $\mathcal W_1 \otimes M_{h^{\perp}}(1)$--module:
$$\mathcal F_{\lambda} \otimes F = \bigoplus_{n \in {\Z}} \mathcal L_{n} (\mathcal F_{\lambda})$$
where
$$ \mathcal L_{n} (\mathcal F_{\lambda}) = \{ v \in  \mathcal F_{\lambda} \otimes F \ \vert    h^{\perp} (0) v  = (n + \Delta) v  \}$$
is an irreducible $\mathcal W_1 \otimes M_{h^{\perp}}(1)$--module.  By using Lemma \ref{hw2} we see that $ \mathcal L_{0} (\mathcal F_{\lambda})$ must be isomorphic to the irreducible highest weight module  $L(x_{\lambda}, y_{\lambda}) \otimes M_{h^{\perp}} (1, \Delta)$.
 Since $\Psi^{-n}(\mathcal W_1) \otimes M_{h^{\perp}}(1,n)$ are simple-current $\mathcal W_1 \otimes M_{h^{\perp}}(1)$--modules we get that 
\bea \mathcal L_{n} (\mathcal F_{\lambda})  &=& \left( \Psi^{-n}(\mathcal W_1) \otimes M_{h^{\perp}}(1,n) \right) \times  \left(  L(x_{\lambda}, y_{\lambda})  \otimes   M_{h^{\perp}} (1, \Delta) \right) \nonumber \\ &=& \Psi^{-n} (L(x_{\lambda}, y_{\lambda})) \otimes  M_{h^{\perp}}(1, \Delta+ n)\eea for every $n \in \mathbb{Z}$. (Here $"\times"$ denotes the fusion product in the category of   $\mathcal W_1 \otimes M_{h^{\perp}}(1)$--modules).

 The proof follows.
\end{proof}

\begin{lemma}
$\mathcal{W}_1$  has a family of irreducible highest weight modules
	$$T_{(1)} :=  \{ L(x, y) \ \vert \ (x,y) \in {\C}^2, \ h_1(x,y)=0 \}$$ which are realized as irreducible quotients of $$ U_{(1)} (\lambda) =\mathcal{W}_1.E^{\lambda}_{1}, \quad \lambda \in \mathbb{C},$$ where $E^{\lambda}_{1}:= \mathbbm{1}^{tw}_{F} \otimes  \mathbbm{1}_{F^{1/2}} \otimes e^{-\mu + \lambda c}$ are highest weight vectors for $\mathcal{W}_1$.
\end{lemma}

\begin{proof}
Consider the $\sigma_F \otimes g$-twisted $F \otimes F^{1/2} \otimes \Pi^{1/2}(0)$-module $\mathcal{F}_{(1)}(\lambda):= M_F^{tw} \otimes F^{1/2} \otimes \Pi^{1/2}_{-1}(\lambda)$, $\lambda \in \mathbb{C}$. Then $\mathcal{F}_{(1)}(\lambda)$ is an untwisted $\mathcal W_1$-module. 

\smallskip

Let $Y^{tw}_F(\omega,z)= Y(\Delta(h,z)\omega,z)$ and set $h= \tfrac{\alpha}{2}$. We have  $$ \Delta(\frac{\alpha}{2}, z)\omega = \omega + \frac{\alpha}{2}z^{-1} + \frac{1}{2}\frac{\alpha}{2}(1)\frac{\alpha}{2}z^{-2}= \omega + \frac{\alpha}{2}z^{-1} + \frac{1}{8}z^{-2}\mathbbm{1}, $$ hence $L(0)\mathbbm{1}^{tw}_{F}= 1/8\mathbbm{1}^{tw}_{F}$.

Similarly, $$ \Delta(\frac{\alpha}{2}, z)\alpha = \alpha + \frac{1}{2}\alpha(1)\alpha(-1) z^{-1}= \alpha + \frac{1}{2}z^{-1}\mathbbm{1}, $$ hence $\alpha(0)\mathbbm{1}^{tw}_{F}= 1/2\mathbbm{1}^{tw}_{F}$.

It holds that 
\begin{align*}
    h(n)E^{\lambda}_{1} &= \delta_{n,0} (-k+ 2\lambda)E^{\lambda}_{1}, &\\
    L_{sug}(n)E^{\lambda}_{1} &= -\frac{5}{16}\delta_{n,0} E^{\lambda}_{1}, \quad  n \in {\Bbb Z}_{\ge 0}. &
\end{align*}
We have
\bea
    &J(0)E^{\lambda}_{1}  &=  \left(\frac{5}{6} - \frac{2}{3}(-k+ 2\lambda)\right) E^{\lambda}_{1},    \nonumber \\
   & L(0)E^{\lambda}_{1} &= \left(-\frac{5}{16}+ \frac{1}{8}+ \frac{1}{3}(\frac{1}{2} - (-k+ 2\lambda))^2 -\frac{5}{12}+  \frac{1}{3}(-k+ 2\lambda) \right) E^{\lambda}_{1}   \nonumber   \\
   & & = \frac{1}{48} (4 (-k + 2 \lambda) -5)  (4 (-k + 2 \lambda) +5)E^{\lambda}_{1} \nonumber 
\eea
Set \bea x:= x_{\lambda} &= &\frac{5}{6} -\frac{2}{3}(-k+ 2\lambda),   \label{weight-tw-1}\\  y:= y_{\lambda}&=& \frac{1}{48} (4 (-k + 2 \lambda) -5)  (4 (-k + 2 \lambda) +5)
%-\frac{5}{16}+ \frac{1}{8}+ \frac{1}{3}(\frac{1}{2} - (-k+ 2\lambda))^2 -\frac{5}{12}+  \frac{1}{3}(-k+ 2\lambda) }_{\tfrac{1}{48} (4 (-k + 2 \lambda) -5)  (4 (-k + 2 \lambda) +5)  } 
\label{weight-tw-2}
 \eea
 so that  $$J(0)E^{\lambda}_{1}= x_{\lambda}E^{\lambda}_{1}, \quad L(0)E^{\lambda}_{1} =y_{\lambda} E^{\lambda}_{1}.$$ Since $y=\frac{3}{4}x^2 - \frac{5}{4}x$, the pair $(x_{\lambda},y_{\lambda}) \in \mathbb{C}^2$ satisfies the relation  $$h_1(x,y)= -3x^2+5x+4y=0.$$

%We compute the action of $L(0)$ on $E^{\lambda}_{1}$:

\smallskip

Hence $\mathcal{W}_1$ has a family of highest weight modules $U_{(1)} (\lambda)$, $\lambda \in \mathbb{C}$,  with highest weights $ (x, y)$. In particular, their irreducible quotients $L(x, y)$ are also modules for $\mathcal{W}_1$.

\end{proof}

Using a twisted variant of Theorem   \ref{ired-general}  and Proposition  \ref{ired-osp-2} we get the following irreducibility result.
\begin{proposition} \label{dec-prva-treca}
Assume that $\lambda \notin {\Z} \cup (-\frac{1}{4} +  {\Z}). $  Then  $\mathcal E_{\lambda}^0 \otimes M_F ^{tw} $ is a completely reducible $\mathcal W_1 \otimes M_{h^{\perp}}(1)$--module:
\bea
\mathcal E_{\lambda}^0 \otimes  M_F ^{tw}  \cong  \bigoplus _{n \in {\Z} } \Psi^{-n} (L(x_{\lambda}, y_{\lambda})) \otimes  M_{h^{\perp}}(1, \Delta'+ n) \label{dec-lambda-tw}
\eea
where  $\Delta'= \frac{1}{2}+k-2\lambda $ and weights $(x_{\lambda}, y_{\lambda})$ are given by  (\ref{weight-tw-1})- (\ref{weight-tw-2}).
In particular, $U_{(1)} (\lambda)$ is an irreducible $\mathcal W_1$-module and it holds that $$ U_{(1)} (\lambda) = L (x_{\lambda} ,y_{\lambda}).$$
\end{proposition}

%Finally, in order to realize the modules $L(x,y)$ such that the  weights $(x,y)\in {\C}^2$ are zeroes of the polynomial $h_3(x,y)$, we use certain symmetry properties of the $\Delta$-operator.

\begin{remark}
Theorem \ref{modules-general} implies that there exists another family of irreducible highest weight $\mathcal{W}_1$-modules $L(x,y)$, for which it holds that $h_3(x,y) = 0$. Indeed, these modules can be obtained from  $T_{(1)}$ using the spectral flow automorphism $\psi^{-1}$ as $$ T_{(3)} :=   \{ \psi^{-1} (L(\hat{x}_1, \hat{y}_1)) \ \vert \ (x,y) \in {\C}^2,  \  h_1(x,y) = 0\}. $$ From Lemma \ref{i+j} it easily follows that  $ T_{(3)} = \{ L(x, y) \ \vert \ (x,y) \in {\C}^2,  \  h_3(x,y) = 0\} .$

These modules also appear in the decomoposition in Proposition \ref{dec-prva-treca}.
\end{remark}

\section*{Acknowledgment}
 We would like to thank D. Ridout for useful  discussions.
The authors  are  partially supported   by the
QuantiXLie Centre of Excellence, a project coffinanced
by the Croatian Government and European Union
through the European Regional Development Fund - the
Competitiveness and Cohesion Operational Programme
(KK.01.1.1.01.0004).

\Addresses

\end{document}